\documentclass{article}

\usepackage[english]{babel}
\usepackage{amssymb}
\usepackage{enumitem}
\usepackage{url}
\usepackage{amsmath,amsthm}
\usepackage{geometry}
\usepackage[hidelinks]{hyperref}

\newtheorem{theorem}{Theorem}[section]
\newtheorem{corollary}[theorem]{Corollary}
\newtheorem{lemma}[theorem]{Lemma}
\newtheorem{proposition}[theorem]{Proposition}

\theoremstyle{definition}
\newtheorem{definition}[theorem]{Definition}
\newtheorem{remark}[theorem]{Remark}

\numberwithin{equation}{section}

\newcommand{\Q}{\mathbb{Q}}

\newcommand{\Z}{\mathbb{Z}}
\newcommand{\C}{\mathbb{C}}
\newcommand{\N}{\mathbb{N}}

\newcommand{\abs}[1]{\left|#1\right|}
\newcommand{\ord}{\operatorname{ord}}
\newcommand{\End}{\operatorname{End}}
\newcommand{\Norm}{\operatorname{Norm}}
\newcommand{\Gal}{\operatorname{Gal}}

\newcommand{\Ker}{\operatorname{Ker}}
\newcommand{\Aut}{\operatorname{Aut}}
\newcommand{\Res}{\operatorname{Res}}
\newcommand{\lcm}{\operatorname{lcm}}

\author{Matteo Verzobio}
\title{Primitive divisors of elliptic divisibility sequences 
	with $j=1728$}
\date{}

\begin{document}
	
	\maketitle
	
		\renewcommand{\thefootnote}{}
	
	
	
	\renewcommand{\thefootnote}{\arabic{footnote}}
	\setcounter{footnote}{0}

		\begin{abstract}
		Take a rational elliptic curve defined by the equation $y^2=x^3+ax$ in minimal form and consider the sequence $B_n$ of the denominators of the abscissas of the iterate of a non-torsion point; we show that $B_{5m}$ has a primitive divisor for every $m$. Then, we show how to generalize this method to the terms in the form $B_{mp}$ with $p$ a prime congruent to $1$ modulo $4$. 
		\end{abstract}

	\section{Introduction}
	\begin{definition}
		A sequence of integers $(x_n)_{n\in \N}$ is a divisibility sequence if
		\[
		m \mid n \implies x_m\mid x_n.
		\]  
		Given a sequence of integers $(x_n)_{n\in \N}$, we say that the $n$-th term has a primitive divisor if there exists a prime $p$ such that
		\[
		p\mid x_n  \text{  and  } p\nmid x_1\cdot x_2\cdots x_{n-1}.
		\]
	\end{definition}
	\begin{definition}
		Take an elliptic curve $E$, defined over $\Q$. Consider $P$ a rational non-torsion point on $E$ and take
		\[
		x(nP)=\frac{A_n}{B_n} \text{   with   } (A_n,B_n)=1 \text{  and  }B_n>0.
		\]
		We will say that the sequence of positive integers $\{B_n\}_{n\in \N}$ is an elliptic divisibility sequence. The sequence of the $B_n$ depends on $E$ and $P$ and we will sometimes denote it with $B_n(E,P)$.
	\end{definition}
	Thanks to \cite[Proposition 10]{silverman}, we know that for every elliptic curve in minimal form and for every non-torsion point $P\in E(\Q)$, $B_n(E,P)$ has a primitive divisor for $n$ large enough. Computational evidence suggests that it is only when $n$ is very small that $B_n$ does not have a primitive divisor. As far as I know, the example of the $B_n(E,P)$ without a primitive divisor for the largest $n$ for $E$ in minimal form is at $n=39$ and it is given at the beginning of page 476 of \cite{ingram2}. Given the curve $E$ defined by the equation
	\[
	y^2+xy+y=x^3+x^2-125615x+61201397
	\]
	and $P=(7107,-602054)$, $B_{39}(E,P)$ does not have a primitive divisor.
	For some classes of curves, there are some effectivity results. For example, in \cite{everestward} it is proved that, if $E$ has a non-trivial rational $2$-torsion point, then $B_n$ has a primitive divisor for $n$ even and greater than an effective computable constant. Also in \cite[Theorem 2.2]{everestward}, an unconditional result for primitive divisors of elliptic divisibility sequences associated with elliptic curves of the form $y^2=x^3-T^2x$ is obtained. The work in \cite{yabutavoutier} both improves and generalizes this result, proving that, if $E$ is defined by $y^2=x^3+ax$ with $a$ fourth-power-free integer, then the sequence of the $B_n$ has a primitive divisor for every $n\geq 3$ even. 
	
	The first aim of this paper is to correct an error in the proof of this fact. In order to prove the main result of \cite{yabutavoutier}, it is necessary to show that $B_{5m}$ has a primitive divisor for every $m$. Yabuta and Voutier prove this in their Lemma 5.1. In the proof of this lemma there is a mistake at the end of page 181 that we want to fix. Putting $C_k=B_{mk}$, the authors assume that if $B_{5m}$ does not have a primitive divisor, then $C_5$ does not have a primitive divisor too. However, this is not necessarily true. It is possible that $C_5$ has a primitive divisor which divides $B_n$ for some $n\mid 5m$ with $m \nmid n$. This means that their use of their Lemma 3.4 to obtain an upper bound for $\log B_{5m}$ is not correct. 
	 This same mistake also seems to affect the proof of Lemma 7 in \cite{ingram2}. We will fix this issue, proving the following.
		\begin{theorem}\label{Teo}
			Let $E_a$ be the elliptic curve defined by the equation \[y^2=x^3+ax\] with $a$ an integer fourth-power-free and $P$ be a non-torsion point in $E(\Q)$. Then, $B_n(E_a,P)$ has a primitive divisor, if $n$ is a multiple of $5$.
		\end{theorem}
	Observe that, up to isomorphism over $\Q$, every elliptic curve in minimal form and with $j$-invariant equal to $1728$ is defined by the equation $y^2=x^3+ax$ with $a$ an integer fourth-power-free.

Finally, we show how to generalize the proof of Theorem \ref{Teo} to the case when we consider the terms in the form $B_{mp}$, for $p\equiv 1\mod 4$, proving the following theorem.
\begin{theorem}\label{thm2}
	Let $E_a$ be as before and $P$ be a non-torsion point in $ E_a(\Q)$. Take a prime $p$ congruent to $1$ modulo $4$. Then $B_n(E_a,P)$ has a primitive divisor if $n$ is square-free, $n>p$ and $p$ is the smallest divisor of $n$.
\end{theorem}
In \cite[Remark 1.5]{yabutavoutier}, it is conjectured that for every sequence $B_n(E_a,P)$, every term has a primitive divisor for $n\geq4$. The work of this paper made one step forward in order to prove this conjecture. 
\section{Preliminaries}
		 We start by recalling the hypothesis of \cite{yabutavoutier} and the facts that we will use. Let $a$ be a fourth-power-free integer and $E_a$ be the elliptic curve defined by the equation $y^2=x^3+ax$. We will denote by $\Delta$ the discriminant of the curve, that is $\Delta:=-64a^3$. We define the height of a rational number as
		 \[
		 H\Big(\frac uv\Big)=\max\{\abs{u},\abs{v}\}
		 \]
		  if $u$ and $v$ are coprime and the logarithmic height as
		 \[
		 h\Big(\frac uv\Big)=\log H\Big(\frac uv\Big).
		 \]
		 Given $P\in E(\Q)$, we define $H(P)=H(x(P))$ and $h(P)=h(x(P))$. We consider the canonical height of a point as defined in \cite[Proposition VIII.9.1]{arithmetic}. Observe that if $P$ is in $E_a(\Q)$ and $x(P)=u/v$ with $u$ and $v$ two coprime integers, then
		 \[
		 y^2=x^3+ax=\frac{u^3+auv^2}{v^3}.
		 \]
		 Since $(u^3+auv^2,v^3)=1$, then $v^3$ is a square and therefore also $v$ it is. In particular, every term of the sequence of the $B_n(E,P)>0$ is a square. Moreover, if $A_n=0$, then $x(nP)=0$ and so $y(nP)^2=x(nP)^3+ax(P)=0$. If $y(nP)=0$, then $2nP=O$ thanks to \cite[III.2.3]{arithmetic} and this is absurd since we are assuming that $P$ has infinite order. So, we have $A_n\neq 0$ for every $n\geq 1$ and in particular $\abs{A_n}\geq 1$.
		 \begin{lemma}\label{C}
		 Let \[C=0.26+\frac{\log \abs{a}}4.\] Then,
		\[
		\abs{h(P)-2\hat{h}(P)}\leq 2C
		\]
		for every $P$ in $E_a(\Q)$. 
	\end{lemma}
\begin{proof}
	This is proved in \cite[Theorem 1.4]{yabutavoutier2}.
\end{proof}
			\begin{lemma}\label{ndp}
			Let $P\in E_a(\Q)$ be a point of infinite order. Consider the elliptic divisibility sequence $(B_n)=(B_n (E_a, P))$. Then,
			\[
			\log B_{5m}\geq 18m^2\hat{h}(P)-29\log \abs{a}-32.863.
			\]
		\end{lemma}
	\begin{proof}
		This was established in the proof of \cite[Lemma 5.1]{yabutavoutier}. See the inequality near the bottom of page 181 in \cite{yabutavoutier}.
	\end{proof}
	\begin{lemma}\label{h}
	Let $a$ be a fourth-power-free integer. For every non-torsion point $P$ in $E_a(\Q)$,
	\[
	\hat{h}(P)\geq \frac{\log\abs{a}-\log 4}{16}.
	\]
	\end{lemma}
\begin{proof}
 This is proved in \cite[Theorem 1.2]{yabutavoutier2}.
\end{proof}
\begin{lemma}\label{rho}
	For every positive integer $n$ define
	\[
	\rho(n)=\sum_{p|n}\frac{1}{p^2}.
	\]
	Then,
	\[
	\rho(n)<\sum_{p \text{  prime  }}\frac{1}{p^2}<0.45225.
	\]
\end{lemma}
	\begin{proof}
		This was proved at the top of page 178 of \cite{yabutavoutier}.
	\end{proof}
\begin{lemma}\label{yv}
	Suppose that $B_n$ does not have a primitive divisor. Then,
	\[
	\log B_n\leq 2\log n+2n^2\rho(n)\hat{h}(P)+2C\omega(n),
	\]
	where $\omega(n)$ is the number of prime divisors of $n$.
\end{lemma}
\begin{proof}
Define
\[
\eta(n)=\sum_{p\mid n}2\log p.
\]
Then, as was proved in \cite[Lemma 3.4]{yabutavoutier},
\[
\log B_n\leq \eta(n)+2n^2\rho(n)\hat{h}(P)+2C\omega(n).
\]
We conclude the proof by observing that
\[
\eta(n)\leq 2\log n.
\]
\end{proof}
\begin{lemma}\label{divpol}	Let $\psi_n$ and $\phi_n$ be the polynomials in $\Z[x,y,a]$ as defined in \cite[Exercise 3.7]{arithmetic}. We recall the properties of these polynomials that we will use in this paper. The $\psi_n$ are the so-called division polynomials.
	\begin{enumerate}[label=(\alph*)]
		\item For every $n>0$ and every $P\in E(\Q)$, 
		\[
		x(nP)=\frac{\phi_n(x(P))}{\psi_n^2(x(P))}.
		\]  
		\item The polynomial $\phi_n$ is in $\Z[x,a]$.
		\item  If $n$ is odd, then the polynomial $\psi_n$ is in $\Z[x,a]$. Instead, if $n$ is even, then $\psi_n$ is a polynomial in $\Z[x,a]$, multiplied by $y$. Therefore, using $y^2=x^3+ax$, we can assume that $\psi_n^2\in \Z[x,a]$ for every $n$.
		\item The polynomial $\phi_n(x)$ is monic and has degree $n^2$. Instead, the polynomial $\psi_n^2(x)$ has degree $n^2-1$ and its leading coefficient is $n^2$. The zeros of this polynomial are the $x$-coordinates of the non-trivial $n$-torsion points of $E(\overline{\Q})$.
	\end{enumerate}
\end{lemma}
\begin{proof}
	See \cite[Exercise 3.7]{arithmetic}.
\end{proof}
Let $x(P)=u/v$ with $(u,v)=1$. We define, with a little abuse of notation, $\phi_n(u,v)$ and $\psi_n^2(u,v)$ the homogenization of the polynomials evaluated in $u$ and $v$.
Then,
\[
x(nP)=\frac{\phi_n(\frac uv)}{{\psi_n^2(\frac uv)}}=\frac{v^{n^2}\phi_n(\frac uv)}{v^{n^2}{\psi_n^2(\frac uv)}}=\frac{\phi_n(u,v)}{v\psi_n^2(u,v)}.
\]
Observe that $\phi_n(u,v)$ and $v\psi_n^2(u,v)$ are both integers and take \begin{equation}\label{defgn}
g_n:=\gcd(\phi_n(u,v),v\psi_n^2(u,v)).
\end{equation}
Therefore,
\begin{equation}\label{psi}
B_n=\frac{v\psi_n^2(u,v)}{g_n}.
\end{equation}
\begin{lemma}\label{phipsi}
	For every $n$ and $m$,
	\[
	\abs{\phi_n\psi_m^2-\psi_n^2\phi_m}^2=\psi_{n+m}^2\psi_{\abs{n-m}}^2.
	\]
\end{lemma}
\begin{proof}
	Observe that both sides have degree $2(n^2+m^2-1)$. The leading term of both sides is $(n^2-m^2)^2$. Then, we just need to check that the zeros of the two polynomials are the same.
	Using the definition, \[\phi_m(x)\psi_n^2(x)-\phi_n(x)\psi_m^2(x)=(x(mP)-x(nP))\psi_n^2(x)\psi_m^2(x).\]
	Thanks to the group law, for every point $R\in E(\Q)$, $x(R)=x(-R)$, as shown for example in \cite[III.2.3]{arithmetic}. If $Q$ is a point of $(n+m)$-torsion, then $x(nQ)=x(-mQ)=x(mQ)$ and so the left side is annihilated in the $x$-coordinates of the $(n+m)$-torsion points. If $Q$ is a point of $\abs{n-m}$-torsion, then $x(nQ)=x(mQ)$ and so the left side is also annihilated in the $x$-coordinates of the $\abs{n-m}$-torsion points. The non-trivial $n+m$-torsion points are $(n+m)^2-1$ and the non-trivial $\abs{n-m}$-torsion points are $(n-m)^2-1$. Therefore, the union of these two sets has $2(n^2+m^2-1)$ elements, that is the degree of the polynomial. So, the roots of both polynomials are the abscissas of the non-trivial $(n+m)$-torsion points and the non-trivial $\abs{n-m}$-torsion points. 
\end{proof}
\begin{lemma}\label{res}
	Let $u$ and $v$ be two coprime integers. Then,
	\[
	\gcd(\phi_k(u,v),v\psi_k^2(u,v))\mid \Delta^{k^2(k^2-1)/6}.
	\]
\end{lemma}
\begin{proof}
	Let $R_k:=\Res(\phi_k(x),\psi_k^2(x))$, where with Res we denote the resultant of the two polynomials. Then, there exist two polynomials $P_k$ and $Q_k$ with integer coefficients such that
	\[
	P_k(x)\phi_k(x)+Q_k(x)\psi_k^2(x)=R_k.
	\]
	Evaluating the equation in $x=u/v$ and multiplying by an appropriate power of $v$, we have
	\[
	P_k'(u,v)\phi_k(u,v)+vQ_k'(u,v)\psi_k^2(u,v)=R_kv^s
	\]
	where $P_k'$ and $Q_k'$ are two bivariate polynomials. Thus,  \[\gcd(\phi_k(u,v),v\psi_k^2(u,v))|R_kv^s.\] If $p$ is a prime divisor of $v$, then
	\[
	\phi_k(u,v)\equiv u^{k^2}\not\equiv 0 \mod(p)
	\]
	since $\phi_k$ is monic and then the gcd does not divide any prime divisor of $v$. So,
	\[
	\gcd(\phi_k(u,v),v\psi_k^2(u,v))|R_k.
	\]
	Using \cite[Theorem 1.1]{disc} we know that 
	\[
	R_k=\Delta^{\frac{k^2(k^2-1)}{6}}
	\]and so we conclude.
\end{proof}
\begin{remark}
	We can assume $\abs{a}\geq 2$. Indeed, if $\abs{a}=1$, then $E_a$ has rank $0$ and so there are no non-torsion points. 
\end{remark}
Now, we briefly show how we perform some of the computations. We will use PARI/GP 2.11.1 \cite{PARI2} and SAGE 8.2 \cite{sagemath}.
\begin{itemize}
	\item  A Thue equation can be solved using the command "thue" on the software PARI/GP.
	\item At some point we will need to compute a lower bound for the canonical height of every non-torsion point of a given curve $E$. We can use the command "E.height\_function ().min(.0001, 20)" of SAGE.
	This gives a lower bound with an error less than $0.01$ for the curves that we will consider. This means that the canonical height of every non-torsion point can be bounded from below by this value minus $0.01$.
	\item With the command "ellheight" of PARI/GP we can compute the canonical height of a point on an elliptic curve.
\end{itemize}
\begin{remark}
	The definition of the canonical height used by PARI/GP and SAGE is slightly different from our definition. Indeed, our canonical height is half the height of the two software. So, every value computed with PARI/GP and SAGE has to be divided by $2$.
\end{remark}
	\begin{lemma}\label{100}
	If $\abs{a}\leq 100$, then 
	\[
	\hat{h}(P)\geq \frac{\log \abs{a}+\log 16}{42}> 0.023\log\abs{a}+0.066
	\]
	for every non-torsion point $P\in E_a(\Q)$.
\end{lemma}
\begin{proof}
	If $a\not\equiv4\mod 16$, then, using \cite[Theorem 1.2]{yabutavoutier2},
	\[
	\hat{h}(P)\geq \frac{\log\abs{a}+4\log 2}{16}>\frac{\log \abs{a}+\log 16}{42}.
	\] 
	If $a\equiv 4\mod 16$ and $\abs{a}\leq 100$, then we have to study only $13$ curves. If $a\neq 68$, then, using the ecdb database \cite{lmfdb}, we find that these curves have rank $0$ or $1$. If the curve has rank $0$, then the lemma is trivial since there are no non-torsion points. If the curve has rank $1$, then the minimum for $\hat{h}(P)$ is at the generator of the curve and then, using the ecdb database we can check that the inequality holds. Observe that the database \cite{lmfdb} uses the definition of $\hat{h}$ as PARI/GP, so every value of $\hat{h}$ taken from the ecdb database has to be divided by $2$. It remains to deal with the curve with $a=68$. Using Sage 8.2, it is possible to find a lower bound for the canonical height over an elliptic curve, using the command "E.height\_function ().min(.0001, 20)". For the curve $E_{68}$ this bound is $\hat{h}\geq 0.32$ and so the inequality still holds. The lowest value for \[\frac{\hat{h}(P)}{\log\abs{a}+\log 16}\] is at $a=-12$ and $x(P)=-2$, where it is $0.2383... >1/42$.
\end{proof}
\begin{lemma}\label{min}
	For every $a$ fourth-power-free and for every non-torsion point $P\in E_a(\Q)$, it holds
	\[
	\hat{h}(P)\geq \frac 1{10}.
	\]
\end{lemma}
\begin{proof}
	If $\abs{a}\geq 100$, then using Lemma \ref{h}
	\[
	\hat{h}(P)\geq \frac{\log 100-\log 4}{16}\geq  \frac 15.
	\]
	If $5\leq \abs{a}\leq 100$, then from Lemma \ref{100},
	\[
	\hat{h}(P)\geq \frac{\log 5+\log 16}{42}\geq  \frac 1{10}.
	\]
	For $1\leq \abs{a}\leq 4$, the ranks of the curves are zero except for $a=-2$ and $a=3$, where the ranks are both 1. These two cases can be checked using the ecdb database.
\end{proof}
	\section{Proof of Theorem \ref{Teo}}
	Recall that we are considering the elliptic divisibility sequence $B_n(E_a,P)$, where $E_a$ is defined by the equation $y^2=x^3+ax$ with $a$ fourth-power-free integer. With the next lemma, we show that we need to prove the theorem only for $n$ square-free. Observe that the next lemma holds for every elliptic curve, not only for $j(E)=1728$. 
	\begin{lemma}\label{rad}
	Let $E$ be a rational elliptic curve and let $P$ be a non-torsion point. Given a natural number $n>1$, let $r:=\prod_{p\mid n}p$ denote the radical of $n$. If $B_n(E,P)$ does not have a primitive divisor, then neither does $B_r(E,(n/r)P)$.
\end{lemma}  
\begin{proof}
	Suppose that $B_r(E,(n/r)P)$ has a primitive prime divisor $q$. We want to show that $q$ is also a primitive divisor for $B_n(E,P)$. Observe that $B_n(E,P)=B_{r}(E,(n/r)P)$. Suppose that $q$ divides $B_{n'}(E,P)$ with $n'$ a proper divisor of $n$. So, there exists a prime $p$ such that $n'\mid n/p$. Using that we are considering divisibility sequences, we have that $q$ divides $B_{n/p}(E,P)$. Take $r'=r/p$. Hence, $B_{r'}(E,(n/r)P)$ is divisible by $q$ since it is equal to $B_{n/p}(E,P)$. This is absurd considering that we assumed that $q$ was a primitive divisor of $B_r(E,(n/r)P)$.
\end{proof}
So by the contrapositive, it follows that if $B_r(E_a,(n/r)P)$ always
has a primitive divisor, then so does $B_n(E_a,P).$ Since $r$ is square-free, then in order to prove Theorem \ref{Teo}, we just need to prove that $B_n(E_a,P)$ has always a primitive divisor for $n$ square-free.
		\begin{proposition}
			Let $n=5m$ be a square-free integer with $m\geq 11$. Then $B_n$ has a primitive divisor. 
		\end{proposition}
		\begin{proof}
			Suppose that $B_n$ does not have a primitive divisor. Then,
			\[
			\log B_n\leq 2\log n+2n^2\rho(n)\hat{h}(P)+2C\omega(n),
			\]
			thanks to Lemma \ref{yv}.
			Therefore, using Lemma \ref{ndp},
			\begin{equation}\label{dis1}
			2n^2\hat{h}(P)\Big(\frac{9}{25}-\rho(n)\Big)\leq29\log \abs{a}+32.863+2C\omega(n)+2\log n.
			\end{equation}
			If we show that the inequality does not hold, then $B_n$ has a primitive divisor.
			Suppose that $2$ does not divide $n$. So,
			\[
			\Big(\frac{9}{25}-\rho(n)\Big)> \Big(\frac{9}{25}-\Big(\sum_{p \text{  prime}} \frac{1}{p^2}-\frac 14\Big)\Big)>0.36-0.45225+0.25>\frac{1}{6.34}
			\]
			and 
			\begin{equation}\label{omega}
			\omega(n)\leq \frac{\log n}{\log 3}< 0.92 \log n
			\end{equation}
			since $3^{\omega(n)}\leq n$.
			Thus, using Lemma \ref{h}, we obtain
			\begin{align*}
			n^2&\leq \frac{6.34(29\log \abs{a}+32.863+2C\omega(n)+2\log n)}{2\hat{h}(P)}\\&\leq 50.8\Big(\frac{2C(0.92 \log n)+29\log \abs{a}+32.863+2\log n}{\log \abs{a}-\log 4}\Big)
			\end{align*}
			and so
			\begin{equation}\label{dis22}
			n^2\leq \log n\Big(\frac{126+23.4\log \abs{a}}{\log \abs{a}-\log 4}\Big)+\Big(\frac{1670+1473.2\log \abs{a}}{\log \abs{a}-\log 4}\Big).
			\end{equation}
			Suppose $\abs{a}\geq 100$. Then,
			\[
			\Big(\frac{126+23.4\log \abs{a}}{\log \abs{a}-\log 4}\Big)\leq 73
			\]
			and
			\[
			\Big(\frac{1670+1473.2\log \abs{a}}{\log \abs{a}-\log 4}\Big)\leq 2627.  
			\]
			Hence, (\ref{dis22}) becomes
			\[
			n^2\leq \log n\Big(\frac{126+23.4\log \abs{a}}{\log \abs{a}-\log 4}\Big)+\Big(\frac{1670+1473.2\log \abs{a}}{\log \abs{a}-\log 4}\Big)\leq 73 \log n+2627.
			\]
			If $n\geq 55$, it is easy to check that
			\[
			n^2\geq 73\log n+2627
			\]
			and so the inequality does not hold and we have a primitive divisor for $\abs{a}>100$ and $n\geq 55$.
			Suppose now $\abs{a}\leq 100$. Then, if $B_n$ does not have a primitive divisor, we know from (\ref{dis1}) that
			\[
			2n^2\hat{h}(P)\Big(\frac{9}{25}-\rho(n)\Big)\leq29\log \abs{a}+32.863+2C\omega(n)+2\log n.
			\]
			Therefore, from Lemma \ref{100} and (\ref{omega}), we obtain
			\begin{align*}
			n^2&\leq \frac{6.34(29\log \abs{a}+32.863+2C\omega(n)+2\log n)}{2\hat{h}(P)}\\&\leq 133.2\Big(\frac{2C(0.92 \log n)+29\log \abs{a}+32.863+2\log n}{\log \abs{a}+\log 16}\Big).
			\end{align*}
			Proceeding as in the case $\abs{a}\geq 100$, we have
			\[
			n^2\leq \Big(\frac{330.2+61.4\log \abs{a}}{\log\abs{a}+\log 16}\Big)\log n+\Big(\frac{3863\log \abs{a}+4377.4}{\log \abs{a}+\log 16}\Big).
			\]
			For $2\leq \abs{a}\leq 100$,
			\[
			\frac{330.2+61.4\log \abs{a}}{\log\abs{a}+\log 16}\leq 108
			\]
			and
			\[
			\frac{3863\log \abs{a}+4377.4}{\log \abs{a}+\log 16}\leq 3005
			\]
			using Lemma \ref{100}.
			It is easy to check that for $n\geq 65$,
			\[
			n^2\geq 108\log n+3005
			\]
			and so the inequality does not holds and we have a primitive divisor for $\abs{a}\leq100$ and $n\geq 65$.
			
			Now, we want to deal with the case $n$ even. We will use the ideas used in the proof of \cite[Theorem 2.4]{everestward}. Let $n=2k$. Then, using (\ref{psi}),
			\[
			x(nP)=x(2(kP))=\frac{\phi_2(x(kP))}{\psi_2^2(x(kP))}=\frac{\phi_2(A_k,B_k)}{B_k\psi_2^2(A_k,B_k)}
			\]
			and therefore
			\begin{equation}\label{Bn}
			B_n=\frac{B_k\psi_2^2(A_k,B_k)}{\gcd(\phi_2(A_k,B_k),B_k\psi_2^2(A_k,B_k))}\geq\frac{4\abs{A_k}B_k(A_k^2+aB_k^2)}{\Delta^2}
			\end{equation}
			since 
			\[
			\gcd(\phi_2(A_k,B_k),B_k\psi_2^2(A_k,B_k))\mid \Delta^{\frac{2^2(2^2-1)}{6}}=\Delta^2,
			\]
			thanks to Lemma \ref{res}.
			If $\abs{A_k}\geq 2\abs{a} B_k$, then put $z=\abs{A_k}/B_k$ and so
			\[
			\abs{A_k^2+aB_k^2}=B_k^2\abs{z^2+a}\geq B_k^2(z^2-\abs{a})\geq B_k^2z=\abs{A_kB_k}
			\]
			since $z\geq 2\abs{a}\geq 4$. If $\abs{A_k}\leq B_k$, then
			\[
			\abs{A_k^2+aB_k^2}\geq \abs{B_k^2}\geq \abs{A_kB_k}
			\]
			and so, in both cases,
			\[
			B_n\geq \frac{4\abs{A_k}^2B_k^2}{\Delta^2}\geq \frac{4H(kP)^2}{\Delta^2}
			\]
			considering that $\abs{A_k}$ and $B_k$ are greater than $1$.
			Otherwise, $B_k\leq \abs{A_k}\leq 2 \abs{a}B_k$ and then $H(kP)=A_k$. Thus, \[B_k\geq \frac{H(kP)}{2\abs{a}}\] and so, using (\ref{Bn}),	
				\[
			B_n\geq\frac{4\abs{A_k}B_k(A_k^2+aB_k^2)}{\Delta^2}\geq \frac{4H(kP)^2}{2\abs{a}\Delta^2}.
			\] 
			Here we are using that $A_k^2+aB_k^2\neq 0$ and so  $\abs{A_k^2+aB_k^2}\geq 1$ since it is a non-zero integer. Indeed, if it is $0$, then $y^2(kP)=x(kP)^3+ax(kP)=0$ and, thanks to the group law, $kP$ would be a $2$-torsion point. This is absurd considering that $P$ is not a torsion point.
			Hence,		\[
			\log B_n\geq \log 4+ 2h(kP)-\log (\abs{2a}\Delta^2)\geq \log 4+4k^2\hat{h}(P)-4C-\log (\abs{2a}\Delta^2)
			\]
			where the second inequality follows from Lemma \ref{C}.
		Therefore, if $B_n$ does not have a primitive divisor, then by Lemma \ref{yv},
			\begin{equation}\label{pari}
			\log 4+ 2n^2\hat{h}(P)\Big(\frac{1}{2}-\rho(n)\Big)\leq 2C(\omega(n)+2)+\log (\abs{2a}(64a^3)^2)+2\log n.
			\end{equation}
			Using Lemma \ref{rho}, we know $\rho(n)< 0.46$. Moreover, $\omega(n)\leq \log n/\log 2$ since $n\geq 2^{\omega(n)}$ and so
			\begin{align*}
			n^2\hat{h}(P)&\leq 12.5\Big(2C(\omega(n)+2)+2\log n+11\log 2+7\log \abs{a}\Big)\\&\leq12.5((2.76+0.73\log \abs{a})\log n+8.67+8\log \abs{a})\\&= 34.5 \log n+9.125 \log \abs{a}\log n+108.375+100\log\abs{a}.
			\end{align*}
			Now, we proceed as in the case odd. If $\abs{a}\geq 100$, then
			\[
			\hat{h}(P)\geq \frac{\log\abs{a}-\log 4}{16}
			\]
			and therefore
			\[
			n^2\leq  \frac{552\log n +146\log n \log\abs{a}+1734+1600\log \abs{a}}{\log\abs{a}-\log 4}.
			\] 
			This equation does not hold for $n\geq 70$ and $\abs{a}\geq 100$. If $2\leq \abs{a}\leq 100$, then
			\[
			\hat{h}(P)\geq \frac{\log\abs{a}+\log 16}{42}
			\]
			and therefore
			\[
			n^2\leq \frac{1449\log n+383.25\log (n)\log\abs{a}+4551.75+4200\log\abs{a}}{\log\abs{a}+\log 16}.
			\]
			This equation does not hold for $n\geq 80$ and $2\leq \abs{a}\leq 100$. So, we have proved the proposition for $n\geq65$ odd and $n\geq 80$ even.  Since we are considering only $n\geq55$ square-free, it remains only the cases $n=55$ and $n=70$. Substituting $n=55$ in (\ref{dis1}) we obtain
			\[
			\hat{h}(P)\leq \frac{41.92+30\log \abs{a}}{1886}.
			\]
			If $\abs{a}\geq 100$, then from Lemma \ref{h},
			\[
			\frac{\log\abs{a}-\log 4}{16}\leq  \frac{41.92+30\log \abs{a}}{1886}
			\]
			and this inequality never holds. If $\abs{a}\leq 100$, then
				\[
			\frac{\log\abs{a}+\log 16}{42}\leq  \frac{41.92+30\log \abs{a}}{1886}
			\]
			and this inequality never holds. So, for $n=55$ there is always a primitive divisor. The case $n=70$ is analogous. We substitute $n=70$ in (\ref{pari}), obtaining
			\[
			\hat{h}(P)\leq \frac{18.73+9.5\log\abs{a}}{1858}.
			\]
			This inequality never holds.
		\end{proof}
	Thanks to Lemma \ref{rad}, we know that we have to prove the theorem for $n$ square-free. We know that $B_n(E_a,P)$ has always a primitive divisor for $n\geq 55$. So, it remains to deal with the cases $n=5,10,15,30$ and $35$. We begin with the case $n=35$. 
\begin{proposition}\label{35}
	The term $B_{n}$ has always a primitive divisor for $n=35$.
\end{proposition}
\begin{proof}
	Suppose that $B_{35}(E_a,P)$ does not have a primitive divisor. Then (\ref{dis1}) must hold for $n=35$. Substituting $n=35$ in (\ref{dis1}), we obtain
	\begin{equation}\label{dis2}
	\hat{h}(P)\leq \frac{30\log \abs{a}+41.1}{734}.
	\end{equation}
	Using Lemma \ref{h},
	\[
	\frac{\log \abs{a}-\log 4}{16}\leq \hat{h}(P)
	\]
	and the inequality 
	\[
	\frac{\log \abs{a}-\log 4}{16}\leq \frac{30\log \abs{a}+41.1}{734}
	\]
	does not hold for $\abs{a}\geq 732$. So, $B_{35}(E_a,P)$ has a primitive divisor if $\abs{a}\geq 732$. Using \cite[Theorem 1.2]{yabutavoutier2}, for all $a$ except those satisfying $a\equiv 4 \mod (16)$,
	\[
	\hat{h}(P)\geq \frac{\log\abs{a}+\log 16}{16}
	\]
	and for this class of curves the inequality
	\[
	 \frac{\log\abs{a}+\log 16}{16}\leq \hat{h}(P)\leq \frac{30\log \abs{a}+41.1}{734}
	\]
	does not hold. In conclusion, $B_{35}(E_a,P)$ has always a primitive divisor except for $a\equiv 4 \mod (16)$ and such that \[732>a>-732.\] Take a point $P$ on $E_a$ such that (\ref{dis2}) holds. From Lemma \ref{C},
	\[
	h(P)\leq 2\hat{h}(P)+2C\leq 0.59\log\abs{a}+0.64.
	\]
	Consider all the couples $(x,a)$ with $x\in \Q$ such that  $a\equiv 4 \mod (16)$, $732>a>-732$, $h(x)\leq 0.59\log\abs{a}+0.64$ and $x^3+ax$ is a rational square. This is a finite set that can be easily compute. If a point $P\in E_a(\Q)$ satisfies (\ref{dis2}), then $(x(P),a)$ must belong to this finite set. Using PARI/GP, we can check for such points if the inequality (\ref{dis2}) holds. It turns out that the only non-torsion points (up to inverse) where the inequality holds is for $(6,36),(30,180)$ in $E_{180}$ and $(6,12)$,$(-2,4)$ in $E_{-12}$. So, we need to check that $B_{35}(E_a,P)$ has a primitive divisor for each of the previous $4$ cases. For such cases, we explicitly compute $B_{35}$ and we check that there is a primitive divisor. In order to do so we use PARI/GP. If $x(P)=u/v$ and $p$ divides $\psi_{35}(u,v)/(\psi_5(u,v)\psi_7(u,v))$ but does not divide $\Delta$, then thanks to (\ref{psi}) this is a primitive divisor of $B_{35}(E_a,P)$. To compute $\psi_{35}(u,v)/(\psi_5(u,v)\psi_7(u,v))$ we use the command "elldivpol" of PARI/GP.
	For example, $139$ is a primitive divisor of $B_{35}(E_{-12},(6,12))$. The other cases are analogous.
\end{proof}
	
	Now, it remains to study the case $n\leq30$. The cases with $n\leq 25$ are proved at the beginning of the proof of \cite[Lemma 5.1]{yabutavoutier}. We want to use the same ideas for the case $n=30$ too, but we need some preliminary lemmas. Recall that $g_n=\gcd(\phi_n(u,v),v\psi_n^2(u,v))$ with $x(P)=u/v$.
	
	The strategy for proving the case $n=30$ is the following. Firstly, we study the sequence of the $g_n$. Secondly, we define a sequence of polynomials $\Psi_n(X,Y)$ and we show that if $B_n$ does not have a primitive divisor, then the equation $\Psi_n(X,Y)=d$ has a solution, for $d$ that ranges in a finite set that depends only on $n$. Finally, we show that the equations $\Psi_{30}(X,Y)=d$ does not have any solution and then $B_{30}$ has a primitive divisor.
\begin{lemma}\label{eq}
	For every $n$ odd,
	\[
	\abs{A_nB_2-A_2B_n}^2= 4^\delta B_{n+2}B_{\abs{n-2}}
	\]
	and for every $n$,
	\[
	\abs{A_nB_4-A_4B_n}^2= B_{n+4}B_{\abs{n-4}}
	\]
	where $\delta\in\{0,1\}$ is a constant that depends only on $E$ and $P$.
\end{lemma}
\begin{proof}
	See \cite[Lemma 3.5]{yabutavoutier}.
\end{proof}
\begin{lemma}
	For every $n$,
	\[
	g_{4+n}g_{\abs{4-n}}= g_4^2g_n^2
	\]
	and for every $n$ odd
	\[
	g_{2+n}g_{\abs{2-n}}= 4^\delta g_2^2g_n^2.
	\]	
\end{lemma}
\begin{proof}
Thanks to (\ref{psi}),
	\[
	g_{n+4}g_{\abs{n-4}} B_{n+4}B_{\abs{n-4}}=v^2\psi_{n+4}^2(u,v)\psi_{\abs{n-4}}^2(u,v)
	\]
	and using Lemma \ref{phipsi}
		\[
	v^2\psi_{n+4}^2(u,v)\psi_{\abs{n-4}}^2(u,v)=\abs{v\phi_n(u,v)\psi_4^2(u,v)-v\psi_n^2(u,v)\phi_4(u,v)}^2.
\]
	Using again (\ref{psi}),
	\[
	\abs{v\phi_n(u,v)\psi_4^2(u,v)-v\psi_n^2(u,v)\phi_4(u,v)}^2=g_n^2g_4^2\abs{A_nB_4-A_4B_n}^2
	\]
	and we conclude using Lemma \ref{eq} since
	\[
	g_n^2g_4^2\abs{A_nB_4-A_4B_n}^2= g_n^2g_4^2B_{n+4}B_{\abs{n-4}}.
	\]
		The other case is analogous.
\end{proof}
\begin{lemma}\label{gn}
	For $n$ odd,
	\[
	g_n=(2^\delta g_2)^\frac{n^2-1}{4}
	\]
	and, for $n\equiv 2 \mod{4}$,
	\[
	g_n=g_2g_4^{\frac{n^2-4}{16}}
	\]
	where $\delta$ is as in Lemma \ref{eq}.
\end{lemma}
\begin{proof}
	We will prove the first equation by induction. Thanks to the definition, $g_1=\gcd(A_1,B_1)=1$ and so the lemma holds for $n=1$. If it holds until $n$, then
	\[
	g_{n+2}=\frac{4^\delta g_2^2g_n^2}{g_{n-2}}=\frac{(2^\delta g_2)^2(2^\delta g_2)^{2\frac{n^2-1}{4}}}{(2^\delta g_2)^\frac{(n-2)^2-1}{4}}=(2^\delta g_2)^{\frac{(n+2)^2-1}{4}}.
	\]
	The other case is analogous.
\end{proof}
\begin{lemma}\label{g2g4}
	Given a prime $p$ and a non-zero integer $x$, we denote with $\ord_p(x)$ the biggest integer $k$ such that $p^k\mid x$. We have
	\[\ord_2(g_4)\leq 5\ord_2(g_2)+6.\]
\end{lemma}
\begin{proof}
	Using \cite[Exercise 3.7]{arithmetic}, we can explicitly compute $\phi_2$, $\psi_2$, $\phi_4$ and $\psi_4$. We have \[\psi_2^2(u,v)=u^3+auv^2,\] \[\phi_2(u,v)=u^4-2au^2v^2+a^2v^4,\] \[\psi_4^2(u,v)=4(u^3+auv^2)(u^6+5au^4v^2-5a^2u^2v^4-a^3v^6)^2,\] and \[\phi_4(u,v)=(u^8-20au^6v^2-26a^2u^4v^4-20a^3u^2v^6+a^4v^8)^2.\]
	Observe that, if $u$ and $a$ are not both even, then
	\[
	\ord_2(g_4)\leq 6
	\]
	since the equation \[\psi_4^2(u,v)\equiv \phi_4(u,v)\equiv 0\mod{64}\]
	does not have non-trivial solutions by direct computation of all the possible cases modulo $64$.
	Define \[k:=\min\{\ord_2(a)/2,\ord_2(u)\}.\] Therefore, by definition $\ord_2(\phi_2(u,v))\geq 3k$ and $\ord_2(\psi_2^2(u,v))\geq 3k$. So,
	\[
	\ord_2(g_2)\geq 3k.
	\]
	If $k=\ord_2(u)\neq\ord_2(a)/2$, then $\ord_2(\psi_4^2)=2+15k$. If $k=\ord_2(a)/2\neq \ord_2(u)$, then $\ord_2(\psi_4^2)=2+14k$. In both cases \[\ord_2(g_4)\leq \ord_2(\psi_4^2)\leq 15k+2.\] It remains the case when $k=\ord_2(a)/2=\ord_2(u)$. Put $a'=a/2^{2k}$ and $u'=u/2^k$ and hence $g_4=2^{15k}\gcd (\phi_4(u',v,a'),\psi_4^2(u',v,a'))$ where with $\phi_4(u',v,a')$ and $\psi_4^2(u',v',a')$ we denote $\phi_4$ and $\psi_4^2$ where we substitute $a$ with $a'$. In this case we can use the previous result on the $\gcd$ in the case when $u$ and $a$ are not even, concluding that
	\[
	\ord_2(\gcd (\phi_4(u',v,a'),\psi_4^2(u',v,a')))\leq 6.
	\]
	In conclusion,
	\[
	\ord_2(g_4)\leq 15k+6\leq 5\ord_2(g_2)+6.
	\]
\end{proof}

The aim of next lemmas is to replicate the work of Ingram in \cite[Section 2]{ingram2}. We want to improve \cite[Lemma 5]{ingram2}. We will use the ideas of Ingram and our work on the sequence of the $g_n$. 

\begin{lemma}\label{erringram}
	Fix $n>2$. Consider the polynomial
	\[
	\Psi_n(x):=\frac{\psi_n(x,y)}{\lcm_{l\mid n}\psi_{n/l}(x,y)}.
	\]
	This polynomial depends only on $x$ and if a prime $p$ divides $\Psi_n(1)\rvert_{a=-1}$, then $p$ divides $2n$. Moreover, if a prime $p$ divides $\Psi_n(0)\rvert_{a=1}$, then $p$ divides $2n$.
\end{lemma}
\begin{proof}
 We start by showing that the polynomial depends only on $x$. If $n$ is odd, then we conclude easily observing that $\psi_k$ depends only on $x$ if $k$ is odd. If $n=2^k$, then $k\geq 2$ for the hypothesis $n>2$ and so $\Psi_n=\psi_n/\psi_{n/2}$. We conclude by using that for $n$ even the polynomial $\psi_n$ is in the form $yp_n(x)$, where $p_n$ depends only on $x$, thanks to part (c) of Lemma \ref{divpol}. If $n=2^kd$ with $d$ odd and greater than $1$, then $y$ divides $\lcm_{l\mid n}\psi_{n/l}(x,y)$ since it divides $\psi_{n/l}(x,y)$ for every prime divisor $l$ of $d$ and then we argue as in the previous case. Define $p_n(x)=\psi_n(x)$ for $n$ odd and $p_n(x)=\psi_n(x,y)/y$ for $n$ even. Hence, 
\[
\Psi_n(x):=\frac{p_n(x)}{\lcm_{l\mid n}p_{n/l}(x)}.
\]
Now we prove that if $p$ divides $p_n(1)\rvert_{a=-1}$, then $p$ divides $2n$. Define $h_k=p_k(1)\rvert_{a=-1}$. Using the recurrence law on the $\psi_n$, we obtain, for $k\geq 1$,
\begin{itemize}
	\item $h_{4k+1}=-h_{2k-1}h_{2k+1}^3$;
	\item $h_{4k+3}=h_{2k+3}h_{2k+1}^3$;
	\item $h_{4k}=\frac{h_{2k}}{2}(h_{2k+2}h_{2k-1}^2-h_{2k-2}h_{2k+1}^2)$ if $k\neq 1$;
	\item $h_{4k+2}=\frac{h_{2k+1}}{2}(h_{2k+3}h_{2k}^2-h_{2k-1}h_{2k+2}^2)$.
\end{itemize}
We briefly show how to obtain the first equality, all the others are analogous. Using \cite[Exercise 3.7]{arithmetic} and the definition of $p_n$, we know that
\begin{align*}
p_{4k+1}(x)&=\psi_{4k+1}(x)\\&=\psi_{2k+2}(x,y)\psi_{2k}^3(x,y)-\psi_{2k-1}(x)\psi_{2k+1}^3(x)\\&=y^4p_{2k+2}(x)p_{2k}^3(x)-p_{2k-1}(x)p_{2k+1}^3(x)\\&=(x^3+ax)^2p_{2k+2}(x)p_{2k}^3(x)-p_{2k-1}(x)p_{2k+1}^3(x).
\end{align*}
Evaluating the equation in $x=1$ and $a=-1$ we obtain 
\[
h_{4k+1}=-h_{2k-1}h_{2k+1}^3.
\]
Now, explicitly writing the first terms on the sequence of the division polynomials, we have $h_1=1$, $h_2=2$, $h_3=-4$ and $h_4=-32$. 
By induction, it is easy to check that $h_{2k}=(-1)^{k-1}k2^{k^2}$ and $h_{2k+1}=(-1)^k2^{k(k+1)}$. For example,
\[
h_{4k+1}=-h_{2k-1}h_{2k+1}^3=(-1)^{4k}2^{k(k-1)+3k(k+1)}=(-1)^{2k}2^{2k(2k+1)}
\] 
where the second equality follows by induction. The other cases are analogous. So, if $p$ divides $p_n(1)\rvert_{a=-1}=h_n$, then it divides $2n$. Therefore, if $p$ divides $\Psi_n(1)\rvert_{a=-1}$, then it divides $2n$.

Now we want to study $\Psi_n(0)\rvert_{a=1}$. Define $j_n=p_n(0)\rvert_{a=1}$, that
satisfies the same recurrence relations as $h_n$. By induction, it is easy to prove that $j_{2k}=(-1)^{k-1}2k$ and $j_{2k+1}=(-1)^k$. For example,
\[
j_{4k+1}=-j_{2k-1}j_{2k+1}^3=(-1)^{1+k-1+3(k)}=1=(-1)^{2k}.
\] Hence, we conclude as before.
\end{proof}
Define the polynomial $F_n(u,v,a)$ as the homogenization of $\Psi_n(x)$, i.e.
\[
F_n(u,v,a)=v^{\deg(\Psi_n(x))}\Psi_n\Big(\frac uv\Big).
\]
So, $F_n \in\Z[u,v,a]$. We put $a$ in the variables to emphasize that $F_n$ depends also on $a$. 
\begin{lemma}
	Let $n>2$. The polynomial $F_n(u,v,a)$ can be written as a homogenous polynomial in the variables $u^2$ and $av^2$. This means that there exists a homogenous polynomial $\Psi_n(X,Y)\in \Z[X,Y]$ such that
	\[
	F_n(u,v,a)=\Psi_n(u^2,av^2).
	\]
\end{lemma}
\begin{remark}
	The definition of $\Psi_n(X,Y)$ is an abuse of notation since we defined before the polynomial $\Psi_n(x)$. Anyway, we did it because the two polynomials are strictly related. Indeed, the polynomial $\Psi_n(X,Y)$ is the homogeneization of $\Psi_n(x)$ composed with a change of variables. Observe that $\Psi_n(x)=\Psi_n(x^2,a)$. For example, $\Psi_3(x)=3x^4+6ax^2-a^2$ and $\Psi_3(X,Y)=3X^2+6XY-Y^2$. We are following the notation used by Ingram in \cite{ingram2}.
\end{remark}
\begin{proof}
	As in the previous lemma, we define $p_n(x)=\psi_n(x)$ for $n$ odd and $p_n(x)=\psi_n(x,y)/y$ for $n$ even. We start by showing that, for all $n\geq 1$, $p_n$ can be written as a homogeneous polynomial with integral coefficients in $x^2$ and $a$. These homogeneous polynomials have degree $(n^2-1)/4$ if $n$ is odd and $(n^2-4)/4$ for $n$ even. For $1\leq n\leq 4$ this follows from the definition. For example,
	\[
	p_3(x)=\psi_3(x)=3x^4+6ax^2-a^2
	\]
	and this can be written as a polynomial of degree $2$ in the variable $x^2$ and $a$. Now, we proceed by induction. By definition,
	\[
	p_{4k+1}(x)=(x^3+ax)^2p_{2k+2}(x)p_{2k}^3(x)-p_{2k-1}(x)p_{2k+1}^3(x)
	\]
	and, by induction, both addends have degree $((4k+1)^2-1)/4$ in the variable $x^2$ and $a$. For example, the degree of the first addend is
	\begin{align*}
	3+\frac{(2k+2)^2-4}{4}+3\frac{(2k)^2-4}{4}&=\frac{12+4k^2+8k+4-4+12k^2-12}{4}\\&=\frac{(4k+1)^2-1}{4}.
	\end{align*} Moreover, every term involved can be written as a homogeneous polynomial with integral coefficients in $x^2$ and $a$ observing that
	\[
	(x^3+ax)^2=x^6+2ax^4+a^2x^2
	\] 
	and using the induction. So, $p_{4k+1}$ can be written as a homogeneous polynomial in the variable $x^2$ and $a$ with degree $((4k+1)^2-1)/4$. The cases $n=4k$, $n=4k+2$ and $n=4k+3$ are analogous. Therefore, $p_n$ can be written as a homogeneous polynomial with integral coefficients in $x^2$ and $a$. We know, thanks to the work in the previous lemma, that
	\[
	\Psi_n(x)=\frac{p_n(x)}{\lcm_{l\mid n}p_{n/l}(x)}
	\]
	for $n\geq 3$ and so $\Psi_n(x)$ can be written as a homogeneous polynomial with integral coefficients in $x^2$ and $a$. Let $\Psi_n(X,Y)\in \Z[X,Y]$ be this polynomial and then
	\[
	\Psi_n(x^2,a)=\Psi_n(x).
	\]
	So, taking the homogeneization,
	\begin{align*}
	\Psi_n(u^2,av^2)&=v^{\deg(\Psi_n(x))}\Psi_n((u/v)^2,a)\\&=v^{\deg(\Psi_n(x))}\Psi_n(u/v)\\&=F_n(u,v,a).
	\end{align*}
\end{proof}
Let $P$ be a rational point on the elliptic curve $E_a$ and put $x(P)=u/v$ with $u$ and $v$ coprime. Let $B_n=B_n(E_a,P)$. Define, as in \cite[Lemma 5]{ingram2}, $X=u^2/(u^2,av^2)$ and $Y=av^2/(u^2,av^2)$. So, for $n\geq 3$,
\[
(u^2,av^2)^{\deg{\Psi_n(x)}/2}\Psi_n(X,Y)=\Psi_n(u^2,av^2)=F_n(u,v,a).
\]
Raising to the square, we have
\begin{equation}\label{Fn}
(u^2,av^2)^{\deg{\Psi_n(x)}}\Psi_n^2(X,Y)=F_n^2(u,v,a)=\frac{\psi_n^2(u,v)}{\lcm_{l\mid n}\psi_{n/l}^2(u,v)}.
\end{equation}
The last equality follows from the fact that the homogenization commutes with the $\lcm$.
\begin{lemma}\label{gn12}
		Let $n\in \N_{\geq 3}$. If $B_{n}$ does not have a primitive divisor, then $\Psi_{n}(X,Y)$ divides $ng_n^{1/2}$,
		where $g_n$ is defined in the equation (\ref{defgn}).
\end{lemma}
\begin{proof}
	Recall that
	\[
	B_n=\frac{v\psi_n^2(u,v)}{g_n}
	\]
	and hence $\psi_n^2(u,v)$ divides $B_n g_n$. Moreover, thanks to (\ref{Fn}), $\Psi_n^2(X,Y)$ divides $\psi_n^2(u,v)$. Consider a prime $q$. If $q$ does not divide $B_n$, then
	\[
	\ord_q(\Psi_n^2(X,Y))\leq \ord_q(\psi_n^2(u,v))\leq \ord_q(B_n g_n)=\ord_q(g_n).
	\]
	If $q$ divides $B_n$, then it divides $B_{n/p}$ for some prime divisor $p$ of $n$ considering that $B_n$ does not have a primitive divisor. So, using \cite[Lemma 3.1]{yabutavoutier}, we have
	\[
	\ord_q(B_n)=\ord_q(B_{n/p})+2\ord_q(p).
	\]Observe that $\Psi_n^2(X,Y)$ divides $\psi_n^2(u,v)/\psi_{n/p}^2(u,v)$ thanks to (\ref{Fn}) since $\psi_{n/p}$ divides $ \lcm_{l\mid n}\psi_{n/l}$. Hence, we have
	\begin{align*}
	\ord_q(\Psi_n^2(X,Y))&\leq \ord_q(\psi_n^2(u,v))-\ord_q(\psi_{n/p}^2(u,v))\\&=\ord_q\Big(\frac{g_n}{g_{n/p}}\Big)+\ord_q\Big(\frac{B_n}{B_{n/p}}\Big)\\&\leq \ord_q(g_n)+2\ord_q(p)\\&\leq  \ord_q(g_n)+2\ord_q(n).
	\end{align*}
	So, for every prime $q$, we have
	\[
	\ord_q(\Psi_n^2(X,Y))\leq \ord_q(g_n)+2\ord_q(n)
	\]
	and then $\Psi_n^2(X,Y)$ divides $n^2g_n$. Observe that $g_n$ is a square since, by definition,
	\[
	g_n=\frac{v\psi_n^2(u,v)}{B_n}
	\]
	and every term involved here is a square.
\end{proof}
\begin{lemma}\label{lemma5}
	Let $n\geq 3$. If $B_{n}$ does not have a primitive divisor and a prime $q$ divides $\Psi_{n}(X,Y)$, then $q$ divides $2n$.
\end{lemma}
\begin{proof}
Take $q$ a prime divisor of $\Psi_n(X,Y)$ and then $q$ divides $ng_n^{1/2}$ for the previous lemma. Suppose that $q$ does not divide $2n$. So, it divides $a$ since the prime divisors of $g_n$ are the prime divisors of $\Delta=-64a^3$, thanks to Lemma \ref{res}. We assume that $q$ divides $a$ but does not divide $2n$ and we find a contradiction. Recall that $X=u^2/(u^2,av^2)$ and $Y=av^2/(u^2,av^2)$. If $q$ does not divide $u$, then does not divide $(u^2,av^2)$. Therefore, $q$ divides $Y$ and does not divide $X$. Hence,
	\[
	\Psi_n(X,Y)\equiv n^*X\mod{q}
	\]
	where $n^*$ is the coefficient of $X^{\deg(\Psi_n(X,Y))}$. It is easy to show that $n^*$ divides $n$. This follows from the fact that the leading coefficient of $\psi_n$ is $n$. Thus, since $(q,n)=1$, we have that $n^*X$ is coprime with $q$, that is absurd considering that $q$ divides $\Psi_n(X,Y)$. So, $q$ divides $u$ and $a$. Observe that, since $q$ divides $u$, then does not divide $v$. If $\ord_q(av^2)>\ord_q(u^2)$, then we conclude as before since $q$ divides $Y$ and does not divide $X$. If $\ord_q(av^2)<\ord_q(u^2)$, then $q$ divides $X$ and does not divide $Y$. Therefore,
	\[
	\Psi_n(X,Y)\equiv \Psi_n(0,1)Y^{\deg(\Psi_n(X,Y))}\mod{q}.
	\]
	Using the definition, \[\Psi_n(0,1)=\Psi_n(0)\rvert_{a=1}\] and hence, from Lemma \ref{erringram}, we have that $q$ does not divide $\Psi_n(0,1)$. So, $q$ does not divide $\Psi_n(X,Y)$, that is absurd. It remains the case $\ord_q(u^2)=\ord_q(a)>0$. Considering that $a$ is fourth-power-free, then $\ord_q(a)\leq 3$ and it is even since it is equal to $2\ord_q(u)$. Hence, $\ord_q(u)=1$ and $\ord_q(a)=2$. Since $P$ belongs to the curve, then
	\[
	u^3+auv^2=v^3(x(P)^3+ax(P))=v^3y(P)^2
	\] 
	and so $u^3+auv^2$ is a square. Therefore, $\ord_q(u^3+av^2u)$ is even. We know
	\[
	u^3+auv^2=u(u^2+av^2)=u(u^2,av^2)(X+Y)
	\]
	and hence
	\[
	\ord_q(u^3+avu^2)=1+2+\ord_q(X+Y).
	\]
	Since the LHS is even, we have that $\ord_q(X+Y)$ is odd. So,
	\[
	X+Y\equiv 0\mod q.
	\]
	Therefore, $X\equiv -Y\mod q$ and then 
	\[
	\Psi_n(X,Y)=X^{\deg(\Psi_n(X,Y))}\Psi_n(1,-1)\mod q.
	\]
	Using again Lemma \ref{erringram}, we conclude.
\end{proof}
Thanks to our computations of the $g_n$, we are now able to improve the exponents of \cite[Lemma 5]{ingram2}. For the convenience of the reader, we write here the result of Ingram.
\begin{lemma}\cite[Lemma 5]{ingram2}
	Let $n\geq 5$ be square-free. Consider the elliptic divisibility sequence $B_n=B_n(E_a,P)$ and suppose that $B_n$ does not have a primitive divisor. Then,
	\[
	\Psi_n(X,Y)= 2^{\alpha}\prod_{\substack{l\mid n}}l^{\beta_l}
	\]
	with $\alpha\leq 2d$ and $\beta_l\leq 3d+1$ with $d=n^2(n^2-1)/4$.
\end{lemma}
\begin{proposition}\label{improve5}
Let $n\geq 3$. Let us consider the elliptic divisibility sequence $B_n=B_n(E_a,P)$ and suppose that $B_n$ does not have a primitive divisor for $n$ square-free. Then,
\[
\Psi_n(X,Y)= 2^{\alpha_2}\prod_{\substack{l\mid n\\l\neq 2}}l^{\alpha_l}
\]
with $\alpha_2\leq \frac{75}4 n^2-59$ and $\alpha_l\leq \frac{45}4 n^2-35$.
\end{proposition}
\begin{proof}
 If $n$ is odd, then
	\[
	\Psi_n(X,Y)\mid ng_n^{1/2}=n(2^\delta g_2)^{\frac{n^2-1}{8}}\mid n(2 \Delta^2)^{\frac{n^2-1}{8}}
	\]
	using Lemma \ref{res}, \ref{gn} and \ref{gn12}. If $n$ is even, in the same way,
	\[
	\Psi_n(X,Y)\mid ng_2^{1/2}g_4^{\frac{n^2-4}{32}}\mid n \Delta(\Delta^{40})^{\frac{n^2-4}{32}}=n\Delta^{\frac{5n^2-16}{4}}.
	\]
	Here we are using that $n$ is square-free and so $n\equiv 2 \mod 4$ if $n$ is even. Since $\Delta=-64a^3$, then
	\[
	\Psi_n(X,Y)|n(4a)^{\frac{15n^2-48}{4}}.
	\]
	Indeed, for $n$ even this is simply the definition and for $n$ odd we have
	\[
	\Psi_n(X,Y)\mid n(2 \Delta^2)^{\frac{n^2-1}{8}}=n(2^{13} a^6)^{\frac{n^2-1}{8}}\mid n(4a)^{\frac{15n^2-48}{4}}.
	\]
	Take $p\neq 2$ a prime. If $p$ does not divide $n$, then $\ord_p(\Psi_n(X,Y))=0$ thanks to the previous lemma. If $p$ divides $n$, then
	\begin{align*}
	\ord_p(\Psi_n(X,Y))&\leq \ord_p(n)+\frac{15n^2-48}{4}\ord_p(a)\\&\leq 1+3\frac{15n^2-48}{4}\\&=\frac{45}4 n^2-35
	\end{align*}	since $a$ is fourth-power-free and $n$ is square-free.
	Moreover, in the same way,
	\[
	\ord_2(\Psi_n(X,Y))\leq 1+5\frac{15n^2-48}{4}=\frac{75}4 n^2-59.
	\]
\end{proof}

\begin{remark}
	The previous Proposition is an improvement of \cite[Lemma 5]{ingram2}, since the exponents grow as $n^2$ and so, for $n$ large enough, they are smaller than the exponents of Lemma 5, that grow as $n^4$.
\end{remark}
Now we are ready to conclude the proof of Theorem \ref{Teo}. We need to show that $B_{30}(E_a,P)$ always has a primitive divisor. We will use our bound on the sequence of the $g_n$ and the work on the divisors of $\Psi_n(X,Y)$.
\begin{lemma}\label{x}
	Let $(B_n(E_a,P))_{n\in \N}$ be an elliptic divisibility sequence and suppose that $B_{30}$ does not have a primitive divisor. Then, 
	\[
	z:=\frac{B_{30}B_{5}B_{3}B_{2}}{B_1B_{15}B_{10}B_{6}}
	\]
	is an integer that divides $30^2$.
\end{lemma}
\begin{proof}
	Thanks to \cite[Lemma 3.1]{yabutavoutier}, if $p$ divides $B_k$, then
	\begin{equation}\label{ord}
	\ord_p(B_{mk})=\ord_p(B_k)+2\ord_p(m).
	\end{equation}
	Take $p$ a prime that does not divide $B_{30}$. So, $\ord_p(z)=0$ since $p$ does not divide any of the terms involved. If $p$ divides $B_{30}$, then it must divide one of the other factors. Suppose $p$ divides $B_1$. Hence, thanks to (\ref{ord}),
	\begin{align*}
	\ord_p(z)&=\ord_p(B_{30})+\ord_p(B_3)+\ord_p(B_2)+\ord_p(B_5)\\&-\ord_p(B_{15})-\ord_p(B_{6})-\ord_p(B_{10})-\ord_p(B_1)\\&=4\ord_p(B_1)+2\ord_p(30)+2\ord_p(3)+2\ord_p(2)+2\ord_p(5)\\&-4\ord_p(B_1)-2\ord_p(15)-2\ord_p(6)-2\ord_p(10)-2\ord_p(1)\\&=0.
	\end{align*}
	If $p$ divides $B_2$ and does not divide $B_1$, then
	\begin{align*}
	\ord_p(z)&=\ord_p(B_{30})+\ord_p(B_3)+\ord_p(B_2)+\ord_p(B_5)\\&-\ord_p(B_{15})-\ord_p(B_{6})-\ord_p(B_{10})-\ord_p(B_1)\\&=2\ord_p(B_2)+2\ord_p(15)-2\ord_p(B_2)-2\ord_p(5)-2\ord_p(3)\\&=0.
	\end{align*}
	The cases when $p$ divides $B_3$ and $B_5$ are analogous. If $p$ divides $B_6$ but does not divide $B_3$ and $B_2$, then
	\begin{align*}
	\ord_p(z)&=\ord_p(B_{30})+\ord_p(B_3)+\ord_p(B_2)+\ord_p(B_5)\\&-\ord_p(B_{15})-\ord_p(B_{6})-\ord_p(B_{10})-\ord_p(B_1)\\&=\ord_p(B_6)+2\ord_p(5)-\ord_p(B_6)\\&=2\ord_p(5).
	\end{align*}
	The cases with $B_{10}$ and $B_{15}$ are analogous. Finally, for every prime $p$,
	\[
	0\leq \ord_p(z)\leq 2\ord_p(2)+2\ord_p(3)+2\ord_p(5)= 2\ord_p(30).
	\]
\end{proof}
\begin{proposition}
	The term $B_{30}$ has always a primitive divisor.
\end{proposition}	
\begin{proof}
Take 
\[
z=\frac{B_{30}B_3B_5B_2}{B_1B_6B_{10}B_{15}}.
\]
Then, using (\ref{psi}),
\begin{align*}
z=&\frac{B_{30}B_3B_5B_2}{B_1B_6B_{10}B_{15}}\\=&\frac{v\psi_{30}^2(u,v)}{g_{30}}\frac{v\psi_{3}^2(u,v)}{g_{3}}\frac{v\psi_{5}^2(u,v)}{g_{5}}\frac{v\psi_{2}^2(u,v)}{g_{2}}\cdot\\\cdot&\frac{g_{15}}{v\psi_{15}^2(u,v)}\frac{g_{10}}{v\psi_{10}^2(u,v)}\frac{g_{6}}{v\psi_{6}^2(u,v)}\frac{g_{1}}{v\psi_{1}^2(u,v)}\\=&\frac{v\psi_{30}^2(u,v)v\psi_5^2(u,v)v\psi_3^2(u,v)v\psi_2^2(u,v)}{v\psi_{15}^2(u,v)v\psi_{10}^2(u,v)v\psi_6^2(u,v)v\psi_1^2(u,v)}\cdot\Big( \frac{g_{15}g_{10}g_{6}g_{1}}{g_{30}g_{5}g_{3}g_{2}}\Big).
\end{align*}
Observe that
\[
\Psi_{30}^2(u^2,av^2)=\frac{\psi_{30}^2(u,v)\psi_5^2(u,v)\psi_3^2(u,v)\psi_2^2(u,v)}{\psi_{15}^2(u,v)\psi_{10}^2(u,v)\psi_6^2(u,v)\psi_1^2(u,v)}
\]
thanks to (\ref{Fn}) and then
\[
\abs{\Psi_{30}(u^2,av^2)}\cdot\sqrt{\Big( \frac{g_{15}g_{10}g_{6}g_{1}}{g_{30}g_{5}g_{3}g_{2}}\Big)}=\sqrt{z}.
\]
If $B_{30}$ does not have a primitive divisor, then $z|(30)^2$ by Lemma \ref{x}. Thus, thanks to Lemma \ref{gn},
\begin{align*}
\Psi_{30}(u^2,av^2)|&30\sqrt{\Big( \frac{g_{30}g_{5}g_{3}g_{2}}{g_{15}g_{10}g_{6}g_{1}}\Big)}\\=&30\frac{g_2^{\frac 12}( g_4)^{\frac{30^2-4}{32}}g_2^{\frac 12}( g_4)^{\frac{2^2-4}{32}}(2^\delta g_2)^{\frac{3^2-1}{8}}(2^\delta g_2)^{\frac{5^2-1}{8}}}{(2^\delta g_2)^{\frac{15^2-1}{8}}g_2^{\frac 12}( g_4)^{\frac{10^2-4}{32}}g_2^{\frac 12}( g_4)^{\frac{6^2-4}{32}}}.
\end{align*}
So, \begin{equation}
\label{g4}\Psi_{30}(u^2,av^2)\mid 30g_4^{24}(2^\delta g_2)^{-24}.
\end{equation}
By direct computation, $\Psi_{30}$ is a homogeneous polynomial of degree $144$. Recall that $X=u^2/(u^2,av^2)$ and $Y=av^2/(u^2,av^2)$. Hence, $X$ and $Y$ are coprime.
Using Lemma \ref{lemma5}, we obtain \[\Psi_{30}(X,Y)|2^{a_1}3^{a_2}5^{a_3}.\]
We know, by (\ref{g4}), that
\[
a_1\leq 2+24(\ord_2(g_4)-\ord_2(g_2))\leq 146+96\ord_2(g_2)
\]
where the last inequality follows from Lemma \ref{g2g4}.
If $2|v$, then $\phi_2(u,v)\equiv u^4\not\equiv0\mod 2$ and hence $\ord_2(g_2)\leq\ord_2 (\phi_2(u,v))=0$. If $2\nmid v$, then
\[
\ord_2(g_2)\leq \ord_2(\gcd(\phi_2(x),\psi_2^2(x)))\leq 3
\]
where the last inequality follows from the fact that the equation \[\phi_2(x)\equiv \psi_2^2(x)\equiv 0\mod 16\] has solution only for $a\equiv 0\mod 16$, that is absurd since $a$ is fourth-power-free. This can be checked evaluating the equation for all the pairs $(x,a)$ modulo $16$. Therefore, 
\[
a_1\leq 146+96\ord_2(g_2)\leq 146+96\cdot 3=434.
\]
The equation $\Psi_{30}(X,Y)\equiv 0\mod{3}$ has no solution for $X$ and $Y$ coprime and so $a_2=0$. This follows from the fact that $\Psi_{30}(x,y)\equiv 0\mod 3$ has only the solution $(x,y)\equiv (0,0)\mod 3$ and this can be checked by direct computation. In the same way $a_3=0$. In conclusion, if  $B_{30}(E_a,P)$ does not have a primitive divisor and $x(P)=u/v$, then $\Psi_{30}(X,Y)|2^{434}$ with $X=u^2/(u^2,av^2)$ and $Y=av^2/(u^2,av^2)$.
Hence, we have to solve finitely many Thue equations of the form
\[
\Psi_{30}(X,Y)=d
\]
with $d|2^{434}$.
Using PARI/GP, we want to show that this equation has solutions only for $X=\pm Y$, $X=0$ or $Y=0$ (we will call them the trivial solutions). 
The polynomial can be factored in four irreducible terms, that we will call $p_1,p_2,p_3,p_4$. The polynomial $p_1$ has degree $16$, the polynomial $p_2$ has degree $32$ and the coefficient of $X^{31}Y$ is $-4256$, the polynomial $p_3$ has degree $32$ and the coefficient of $X^{31}Y$ is $-416$ and the polynomial $p_4$ has degree $64$. This factorization is obtained using PARI/GP. If $\Psi_{30}(X,Y)|2^{434}$, then $p_1(X,Y)=\pm 2^k$ for $0\leq k\leq 217$ or $p_2(X,Y)=\pm 2^k$ for $0\leq k\leq 217$. Using PARI/GP one can check that these two equations have only trivial solutions. This calculation took 3 minutes using PARI/GP 2.11.1 on a Windows 10 desktop with an Intel i7-7500 processor and 8gb of RAM. If $X=0$, then $u=0$ and hence $P$ is a 2-torsion point, that is absurd since we assumed that $P$ is non-torsion. If $Y=0$, then $a=0$ or $v=0$ and neither of which we consider here. If $X=-Y$, then $u^2=- av^2$. Therefore $x(P)^2=-a$ and so again $P$ is a $2$-torsion point. If $X=Y$, then $x(P)^2=a$ and hence $y(P)^2=x(P)^3+ax(P)=2a^{3/2}$. Thus, $2a^{1/2}=(y(P)/x(P))^2$ is a square. Suppose $p\neq 2$ divides $a$. Therefore $\ord_p(a)=2\ord_p(2a^{1/2})\geq 2\cdot 2$ and this is absurd since $a$ is fourth-power-free. Hence, it remains only the case when $a$ is a power of $2$. Since $2a^{1/2}$ is a square and $a$ is fourth-power-free, then $a=4$. This curve has rank $0$. So, there is always a primitive divisor.
\end{proof}
\section{Application to a more general case}
We want to generalize the techniques of the previous section to a more general case. The proof of Theorem \ref{thm2} will follow from the ideas of \cite[Lemma 5.1]{yabutavoutier}. As the authors of \cite{yabutavoutier} pointed out in Remark 5.2, $\psi_k$ is reducible for $k=13$ and $17$. We will show that indeed it is reducible for every prime congruent to $1$ modulo $4$ and therefore we will apply their ideas to prove Theorem \ref{thm2}.

Recall that $E_a$ is the elliptic curve defined by the equation $y^2=x^3+ax$ where $a$ is a fourth-power-free integer.
\begin{lemma}\label{rootsep}
	Fix $k\geq 2$.
	Let $T_1$ and $T_2$ be two non-trivial $k$-torsion points of $E_a(\overline{\Q})$ such that $T_1\neq \pm T_2$. There exists $K_k$, depending only on $k$, such that
	\[
	\abs{x(T_1)-x(T_2)}\geq K_k\abs{a}^{-1/2}
	\]
	and
	\[
	\abs{x(T_1)^{-1}-x(T_2)^{-1}}\geq K_k\abs{a}^{-1/2}.
	\]
	In the second inequality we are assuming that $x(T_1)x(T_2)\neq 0$.
\end{lemma}
\begin{proof}
	Fix $k$ and consider the curve $E_1$ defined by the equation $y^2=x^3+x$. Define
	\begin{equation}\label{eqKk}
	K_k:=\min_{\substack{T_1\neq \pm T_2\\T_1,T_2\in E_1(\overline{\Q})[k]\setminus\{O\}}}\{\abs{x(T_1)-x(T_2)},\abs{x(T_1)^{-1}-x(T_2)^{-1}}\}
	\end{equation}
	where $E_1(\overline{\Q})[k]$ is the set of the $k$-torsion points of $E_1(\overline{\Q})$.
	There is a complex isomorphism $\varphi$ between $E_{1}$ and $E_a$ given by the map
	\[
	\varphi:(x,y)\to (a^{\frac 12}x,a^{\frac 34}y).
	\] If $R_1$ is a $k$-torsion point of $E_1$, then $\varphi(R_1)$ is a $k$-torsion point of $E_a$. So, if $R_1$ and $R_2$ are two $k$-torsion points of $E_a(\overline{\Q})$, then
	\[
	x(R_1)-x(R_2)={a}^{1/2}x(T_1)-{a}^{1/2}x(T_2)
	\]
	for $T_1$ and $T_2$ two $k$-torsion points of $E_1(\overline{\Q})$. Therefore,
	\begin{align*}
	&\min_{\substack{T_1\neq \pm T_2\\T_1,T_2\in E_a^\text{tors}(\overline{\Q})[k]\setminus\{O\}}}\{\abs{x(T_1)-x(T_2)}\}\\=& \abs{a}^{1/2} \min_{\substack{T_1\neq \pm T_2\\T_1,T_2\in E_1^\text{tors}(\overline{\Q})[k]\setminus\{O\}}}\{\abs{x(T_1)-x(T_2)}\}
	\end{align*}
	and
	\begin{align*}
	&\min_{\substack{T_1\neq \pm T_2\\T_1,T_2\in E_a^\text{tors}(\overline{\Q})[k]\setminus\{O\}}}\Big\{\abs{\frac1{x(T_1)}-\frac1{x(T_2)}}\Big\}\\=& \min_{\substack{T_1\neq \pm T_2\\T_1,T_2\in E_1^\text{tors}(\overline{\Q})[k]\setminus\{O\}}}\Big\{\abs{\frac1{x(T_1)}-\frac1{x(T_2)}}\Big\}\cdot\abs{a}^{-1/2}.
	\end{align*}
	Again we are assuming that $x(T_1)x(T_2)\neq 0$ in the last equation.
\end{proof}
\begin{lemma}\label{red}
	Suppose that $\psi_k^2(x)=f(x)g(x)$ with $g$ and $f$ two coprime polynomials with integer coefficients. Let $u$ and $v$ be two coprime integers with $v\geq 1$ and $\psi_k^2(u,v)$ be the homogenization of $\psi_k^2(x)$ evaluated in $u$ and $v$. Let $d=\min\{\deg f,\deg g\}$ and suppose $\max\{\abs{u},v\}\geq (K_k\abs{a}^{-1/2}/2)^{-1}$. Therefore,
	\[
	\psi_k^2(u,v)\geq \max\{\abs{u},v\}^d\Big(\frac{K_k\abs{a}^{-1/2}}{2}\Big)^d,
	\]
	if $\psi_k^2(u,v)\neq 0$.
\end{lemma}
\begin{proof}
	Define $x=u/v$. Suppose that $\abs{u}\leq v$ and that the root $\zeta$ of $\psi_k^2$ closest to $u/v$ is a root of $g$. Then, $\abs{x-x_0}\geq K_k\abs{a}^{-1/2}/2$ for every root $x_0$ of $f$. Indeed, otherwise,
	\[
	\abs{\zeta-x_0}\leq\abs{\zeta-x}+\abs{x-x_0}\leq 2\abs{x-x_0}< K_k\abs{a}^{-1/2}
	\] and this is absurd thanks to Lemma \ref{rootsep} since $\zeta$ and $x_0$ are different abscissas of non-trivial $k$-torsion points. Observe that $x_0\neq \zeta$ considering that $x_0$ is a root of $f$, $\zeta$ is a root of $g$ and $(f,g)=1$. Denote with $f(u,v)$ the homogenization of $f$, evaluated in $u$ and $v$. Hence,
	\[
	\abs{f(u,v)}=v^{\deg f}\abs{f(x)}\geq \Big(v\frac{K_k\abs{a}^{-1/2}}{2}\Big)^{\deg f}\geq \Big(\max\{\abs{u},v\}\frac{K_k\abs{a}^{-1/2}}{2}\Big)^{d}.
	\]
	Here, we are using that $\max\{\abs{u},\abs{v}\}\geq (K_k\abs{a}^{-1/2}/2)^{-1}$.
	Observe that $g(u,v)$ is a non-zero integer since $0\neq \psi_k^2(u,v)=f(u,v)g(u,v)$ and then $\abs{g(u,v)}\geq 1.$ In conclusion,
	\[
	\psi_k^2(u,v)=\abs{f(u,v)g(u,v)}\geq \max\{\abs{u},v\}^d\Big(\frac{K_k\abs{a}^{-1/2}}{2}\Big)^d.
	\]
	If $\abs{u}\leq v$ and the root of $\psi_k^2$ closest to $u/v$ is a root of $f$, then the proof is identical.
	
	Suppose now $\abs{u}\geq v$ and that the root $\zeta\neq 0$ of $\psi_k^2$ that minimize $\abs{x^{-1}-\zeta^{-1}}$ is a root of $g$. Then, using again the triangle inequality, \[\abs{x^{-1}-x_0^{-1}}\geq K_k\abs{a}^{-1/2}/2\] for every root $x_0\neq 0$ of $f$. Therefore,
	\[
	\abs{f(u,v)}\geq \abs{u}^{\deg f}\Big(\frac{K_k\abs{a}^{-1/2}}{2}\Big)^{\deg f}\geq  \Big(\frac{\max\{\abs{u},v\}K_k}{2\abs{a}^{1/2}}\Big)^{d}.
	\]
	As above, we have $\abs{g(u,v)}\geq 1$. In conclusion,
	\[
	\psi_k^2(u,v)\geq \max\{\abs{u},v\}^d\Big(\frac{K_k\abs{a}^{-1/2}}{2}\Big)^d.
	\]
	The case when $\abs{u}\geq v$ and the root $\zeta\neq 0$ of $\psi_k^2$ that minimize $\abs{x^{-1}-\zeta^{-1}}$ is a root of $f$ is identical.
\end{proof}
Let $P$ be a non-torsion point of an elliptic curve $E_a(\Q)$. Consider the sequence $B_n=B_n(E_a,P)$. We will focus on the study of the terms $B_{mk}$ for a fixed $k$.
\begin{proposition}\label{Bmk}
	Let us fix $k$ such that $\psi_k^2=f(x)\cdot g(x)$ with $f$ and $g$ two coprime polynomials in $\Z[x]$. Let $d=\min\{\deg f,\deg g\}$. Suppose that $B_{mk}$ does not have a primitive divisor, that $d>k^2\rho(mk)$ and that
	\[
	2m^2\geq \frac{-\log(K_k\abs{a}^{-1/2}/2)+2C}{\hat{h}(P)}.
	\] Then,
	\begin{equation}\label{2mk}
	2(mk)^2\leq \frac{\log\abs{g_k}-d\log \Big(\frac{K_k\abs{a}^{-1/2}}{2}\Big)+2dC+2\log mk+2C\omega(mk)}{(\frac{d}{k^2}-\rho(mk))\hat{h}(P)},
	\end{equation}
	with
	\[
	g_k:=\gcd(\phi_k(A_m,B_m),\psi_k^2(A_m,B_m)).
	\]
\end{proposition}
\begin{remark}
	If we fix $k$ and we let $m$ grow, then the inequality of the proposition does not hold, since the LHS is quadratic in $m$ and the RHS is logarithmic. So, if we take $k$ so that $\psi_k^2$ is reducible and $d>k^2\rho(mk)$, then $B_{mk}$ has a primitive divisor for every $m$ large enough. This is the key point for the proof of Theorem \ref{thm2}.
\end{remark}
\begin{proof}
	Thanks to the hypotheses, $H(mP)\geq( K_k\abs{a}^{-1/2}/2)^{-1}$ since
	\[
	h(mP)\geq 2m^2\hat{h}(P)-2C\geq -\log (K_k\abs{a}^{-1/2}/2).
	\]
	Moreover, $\psi_k^2(A_m,B_m)\neq 0$ since $P$ is a non-torsion point.
	So, we can apply Lemma \ref{red} to $\psi_k^2$.
	Using (\ref{psi}), Lemma \ref{red} and that $B_m\geq 1$, we have
	\[
	B_{mk}=\frac{B_m\psi_k^2(A_m,B_m)}{\gcd(\phi_k(A_m,B_m),\psi_k^2(A_m,B_m))}\geq \frac{H(mP)^d\Big(\frac{K_k\abs{a}^{-1/2}}{2}\Big)^d}{\abs{g_k}}.
	\]
	Considering the logarithms,
	\begin{align*}
	2dm^2\hat{h}(P)&\leq dh(mP)+2dC\\&\leq \log B_{mk}+\log\abs{g_k}-d\log \Big(\frac{K_k\abs{a}^{-1/2}}{2}\Big)+2dC.
	\end{align*}
	Thanks to Lemma \ref{yv}, if $B_{mk}$ does not have a primitive divisor, then
	\[
	\log B_{mk}\leq 2\log mk+2(mk)^2\rho(mk)\hat{h}(P)+2C\omega(mk)
	\]
	and so
	\begin{align*}
	2m^2k^2\Big(\frac{d}{k^2}-\rho(mk)\Big)\hat{h}(P)\leq& \log\abs{g_k}-d\log \Big(\frac{K_k\abs{a}^{-1/2}}{2}\Big)\\+&2dC+2\log mk+2C\omega(mk).
	\end{align*}
\end{proof}
\begin{remark}
	If $k=m_1m_2$ with $(m_1,m_2)=1$, then $f:=\psi_{m_1}^2\psi_{m_2}^2$ divides $\psi_k^2$. Anyway, we can never apply the previous proposition using only this observation. Indeed, in this case, the hypothesis $d>k^2\rho(mk)$ fails. We have $d\leq \deg(f)=m_1^2-1+m_2^2-1$ and \[
	\rho(mk)\geq  \rho(k)=\rho(m_1)+\rho(m_2)\geq \frac{1}{m_1^2}+\frac{1}{m_2^2}.
	\]
	So,
	\[
	d\leq m_1^2-1+m_2^2-1< m_1^2+m_2^2=k^2\Big(\frac{1}{m_1^2}+\frac{1}{m_2^2}\Big)\leq k^2\rho(mk).
	\]
\end{remark}
Now, we want to find some cases where $\psi_p$ is reducible for $p$ a prime. We will show that this happens for a lot of primes. We briefly recall some classical facts on the elliptic curves, for the details see \cite[Section III.9]{arithmetic}. We will denote by $\End(E)$ the ring of the endomorphisms (defined over $\overline{\Q}$) of $E$. Since the map given by the multiplication by a rational integer is an endomorphism, it follows that $\Z\subseteq \End(E)$.  If $\End(E)\setminus \Z$ is not empty, we say that $E$ has complex multiplication. Consider the embedding $\End(E)\hookrightarrow \End(E)\otimes_\Z\Q$. If $E$ has complex multiplication, then every element $\varphi$ of $\End(E)$ can be written as $a+\gamma b$ with $a$ and $b$ in $\Q\subseteq  \End(E)\otimes_\Z\Q$ and $\gamma$ such that $\gamma^2<0$ and $\gamma^2\in \Q$. We say that an endomorphism $\varphi$ splits if there exist $\alpha$ and $\beta$ that are not isomorphisms so that $\varphi=\alpha\beta$. We define the norm as in \cite[Section III.9]{arithmetic}. If $\Norm(\varphi)=1$, then $\varphi$ is an isomorphism and $\Norm(n)=n^2$ for $n\in \Z$. Moreover, $\Norm(\alpha \beta)=\Norm(\alpha)\Norm(\beta)$ for every $\alpha$ and $\beta$ in $\End(E)$. The group $\Gal(\overline{\Q}/\Q)$ has a natural action over $\End(E)$.
\begin{lemma}\label{2p-2}
Suppose that $E$ is a rational elliptic curves that has complex multiplication (not necessarily with $j(E)=1728$). Take $p\neq 2$ a prime that splits in $\End(E)$. Hence, $\psi_p^2(x)=f(x)g(x)$ with $f$ and $g$ two coprime polynomials with integer coefficients and \[d=\min\{\deg f,\deg g\}=2p-2. \]
\end{lemma}
\begin{proof}
 Let $p$ be a prime that splits in $\End(E)$, i.e. there exist $\alpha$ and $\beta$ in $\End(E)\setminus\Aut(E)$ such that
\[
\alpha \beta=p.
\]
 Then $\alpha=a+\gamma b$ with $\gamma$ defined as before. Let $\overline{\alpha}=a-b\gamma$ and take
 $g\in \Gal(\overline{\Q}/\Q)$. Hence,
\[
(\gamma^g)^2=(\gamma^2)^g=\gamma^2
\] 
since $\gamma^2$ is rational
and we conclude $\gamma^g=\pm \gamma$. So, $\alpha^g=a^g+\gamma^gb^g=a\pm\gamma b$ that is $\alpha$ or $\overline{\alpha}$. We denote with $E[p]$ the set of the $p$-torsion points of $E(\overline{\Q})$ that has $p^2$ elements and it is invariant under the action of $\Gal(\overline{\Q}/\Q)$. Consider the set$\{ \alpha(E[p])\cup\overline{\alpha}(E[p])\}$; we will show that it is invariant under the action of the Galois group and has $2p-1$ elements. If $P=\alpha(Q)$, then $P^g=\overline{\alpha}(Q^g)$ or $P^g=\alpha(Q^g)$ and then $P^g$ is in the set $\{\alpha(E[p])\cup\overline{\alpha}(E[p])\}$. Therefore, the set $\{ \alpha(E[p])\cup\overline{\alpha}(E[p])\}$ is invariant under the action of $\Gal(\overline{\Q}/\Q)$.
Observe that $\Norm(\alpha)=p$ since it must be a divisor of $p^2$ and cannot be $1$ and $p^2$ since $\alpha$ and $\beta$ are not isomorphisms. So, \[\Norm(\alpha \overline{\alpha})=\Norm(\alpha)\Norm(\overline{\alpha})=p^2\] and then $\alpha\overline{\alpha}$ is a rational with norm $p^2$, that can be only $p$ or $-p$. Moreover, for the properties of the Norm, $\#\Ker(\alpha)=\Norm(\alpha)=p$ and $\overline{\alpha}(E[p])\subseteq \Ker(\alpha)$ since \[O=\pm p(E[p])=\alpha(\overline{\alpha}(E[p])).\] 
Suppose that $\Ker(\alpha)\cap \Ker(\overline{\alpha})\neq \{O\}$. Take $P\neq O$ in the intersection. Hence $2a(P)=\alpha(P)+\overline{\alpha}(P)=O$ and in the same way $2b(P)=O$, that implies $p|\Norm(2a),\Norm(2b)$. Since $2a$ and $2b$ are both integers, then $\Norm(2a)$ and $\Norm(2b)$ are both squares. So, $p^2$ must divide their norm and then \[p^2|\Norm(2\alpha)=4p,\] that is absurd. Hence, $\Ker(\alpha) \cap \Ker(\overline{\alpha})=\{O\}$ and therefore $\overline{\alpha}(E[p])$ and ${\alpha}(E[p])$ have trivial intersection. The map $\alpha:E[p]\to E[p]$ has kernel with $p$ elements and then the image has $p$ elements (recall that $E[p]$ has $p^2$ elements). We conclude that $\{ \alpha(E[p])\cup\overline{\alpha}(E[p])\}$ has $2p-1$ elements. So, if $f$ is the polynomial with roots the abscissas of the points of the set $\{\alpha(E[p])\cup \overline{\alpha}(E[p])\}\setminus\{O\}$, then $f$ has degree $2p-2$ and $f\in \Q[x]$ since the set of the roots is invariant under the action of $\Gal(\overline{\Q}/\Q).$ We conclude by observing that the roots of $f$ are roots of $\psi_p^2$ and then $f$ divides $\psi_p^2$. 

Let $g=\psi_p^2/f$. The polynomial $g$ has integral coefficients and has degree $p^2-1-(2p-2)=p^2-2p+1$. So, $d=\min\{\deg f,\deg g\}=2p-2$ considering that 
\[
p^2-2p+1\geq 2p-2
\]
for $p\geq 3$. It remains to prove that the two polynomials are coprime. Otherwise, there exists $x_0\in \overline{\Q}$ that is a root of $f$ and $g$. Hence, there are $P, P'\in E[p]\setminus\{O\}$ such that $x_0=x(P)$ with $P\in \{\alpha(E[p])\cup \overline{\alpha}(E[p])\}$ and $x_0=x(P')$ with $P'\notin \{\alpha(E[p])\cup \overline{\alpha}(E[p])\}$. Since $x(P)=x_0=x(P')$, then $P=-P'$ and this is absurd since $\{\alpha(E[p])\cup \overline{\alpha}(E[p])\}$ is invariant under the multiplication by $-1$.
\end{proof}
\begin{corollary}
	Take $p\neq 2$ a prime that splits in $\End(E)$. Then, $\psi_p(x)$ is reducible.
\end{corollary}
\begin{proof}
	Let $\psi_p^2(x)=\prod_i p_i(x)^{a_i}$ be the factorization in prime factors of $\psi_p^2$. Thanks to the previous lemma we know that $\psi_p^2$ has at least two prime divisors. Since $\psi_p^2$ is the square of $\psi_p$, then $a_i$ is even for every $i$ and so $a_i=2b_i$. So,
	\[
	\psi_p(x)=\pm \prod_i p_i(x)^{b_i}
	\]
	and then $\psi_p$ is reducible.
\end{proof}
We use the previous lemma to prove that $\psi_p(x)$ is reducible when $p\equiv 1 \mod 4$, if we consider the curve $E_a$.
\begin{lemma}\label{split}
	Let $p$ be a prime congruent to $1$ modulo $4$. So, $p$ splits in the ring $\End(E_a)$.
\end{lemma}
\begin{proof}
	Let $i$ be the endomorphism of $E_a$ so that
	\[
	i(x,y)=(-x,iy).
	\]
	This is an endomorphism since
	\[
	(iy)^2=-y^2=-x^3-ax=(-x)^3+a(-x)
	\]
	and then the points in the image of the map are still in $E_a$.
	Observe that $i^2=[-1]$, where $[-1]$ is the inverse endomorphism. This shows that $E_a$ has complex multiplication. Thanks to the Fermat's theorem on sums of two squares, we know that there exist $a$ and $b$ integers so that
	\[
	a^2+b^2=p.
	\]
	Here we are using the hypothesis $p\equiv 1 \mod{4}$. Let $\varphi_1$ be the endomorphism $a+ib$ and $\varphi_2$ be the endomorphism $a-ib$. Hence,
	\[
	\varphi_1\varphi_2=(a+ib)(a-ib)=a^2-(ib)^2=a^2+b^2=p.
	\]
	Since $\varphi_1$ and $\varphi_2$ are conjugate, then they have the same norm. So,
	\[
	p^2=\Norm{p}=\Norm{\varphi_1}\Norm{\varphi_2}=\Norm{\varphi_1}^2=\Norm{\varphi_2}^2
	\]
	and then $\varphi_1$ and $\varphi_2$ are not isomorphisms. Therefore, $p$ splits.
\end{proof}
	We sum up the present situation: Thanks to Lemma \ref{2p-2} and \ref{split}, we know that $\psi_p$ is reducible for $p\equiv 1 \mod{4}$. We want to show that $B_{mp}$ has a primitive divisor, under some hypothesis on $m$. Using Proposition \ref{Bmk}, we know that, in order to prove that $B_{mp}$ has a primitive divisor, we have to show that an inequality does not hold. In the inequalities of Proposition \ref{Bmk} every term that appears has been studied, except for $K_k$. So, in the next pages, we will compute an explicit bound for $K_k$ as defined in Lemma \ref{rootsep}.
	
	 Consider the lattice $\Lambda\subseteq \C$ generated by $1$ and $i$. As is shown in \cite[Chapter VI]{arithmetic} there is an isomorphism $\varphi$ between $\C/\Lambda$ and $E(\C)$ given by the map
\[
\varphi(z)=(\wp(z),\frac{\wp'(z)}{2},1)
\]
where
\[
\wp(z)=\frac{1}{z^2}+\sum_{\omega\in \Lambda\setminus 0}\frac{1}{(z-\omega)^2}-\frac{1}{\omega^2}.
\] 
The curve $E$ is defined by the equation $y^2=x^3-15G_4x-35G_6$ with 
\[
G_4:=\sum_{\omega\in \Lambda\setminus 0}\frac{1}{\omega^4}
\]
and
\[
G_6:=\sum_{\omega\in \Lambda\setminus 0}\frac{1}{\omega^6}.
\] 
Thanks to \cite[Proposition VI.3.6]{arithmetic} and \cite[Exercise 6.6]{arithmetic}, we have $j(E)=1728$. So, $G_6=0$ and hence $E$ is defined by the equation 
\[
y^2=x^3-15G_4x.
\]
Define
\[
\sigma(z):=z\prod_{\omega\in \Lambda\setminus 0}\Big(1-\frac z\omega\Big)e^{\frac z\omega+\frac 12(\frac z\omega)^2}
\]and  \[\Lambda^*:=\{\lambda\in\Lambda |\abs{\lambda}>2\}.\] For the details on $\sigma$, see \cite[Lemma VI.3.3]{arithmetic}. Thanks to \cite[Exercise 6.3]{arithmetic}, given $z_1$ and $z_2$ two complex numbers, we have
\[
x(\varphi(z_1))-x(\varphi(z_2))=-\frac{\sigma(z_1+z_2)\sigma(z_1-z_2)}{\sigma(z_1)^2\sigma(z_2)^2}.
\]
For this reason, in order to compute $K_k$, we need to study the function $\sigma$.

We start by computing an upper and a lower bound for the absolute value of $\sigma$ evaluated at $z$, for $z\in \C$ such that $nz\in \Lambda$ and such that $0$ is the element of $\Lambda$ closest to $z$. So, $\varphi(z)$ is a $n$-torsion point of $E(\C)$. In order to do so, we need a preliminary lemma.
\begin{lemma}\label{sum}
	Let $k\geq 3$. Then,
	\[
	\sum_{\omega\in \Lambda^*}\abs{\frac{1}{\omega^k}}\leq \frac {21}{k}.
	\]
\end{lemma}
\begin{proof}
	Every element of $\Lambda^*$ can be written as $a+ib$ with $a$ and $b$ two integers such that $a^2+b^2\geq 5$.
	Fix $a$ and $b$ two strictly positive integers such that $a^2+b^2\geq 5$. Consider the square $Q_{a,b}$ on the complex plane with vertices $(a-1,b-1)$, $(a-1,b)$, $(a,b-1)$ and $(a,b)$. Observe that, for $x\in Q_{a,b}$, $\abs{x}\leq \abs{a+ib}$ and then
	\[
	\frac{1}{\abs{a+ib}^k}\leq \frac{1}{\abs{x}^k}.
	\]
	Moreover, the intersection of two different squares has measure $0$.
	For every $a$ and $b$, if $x\in Q_{a,b}$, then $\abs{x}\geq 1$. Therefore,
	\[
	\frac{1}{\abs{a+ib}^k}\leq \int_{Q_{a,b}}\frac{1}{\abs{x}^k} dx
	\]
	and
	\begin{equation}\label{intxk}
	\sum_{\substack{a,b>0\\a^2+b^2\geq 5}}\frac{1}{\abs{a+ib}^k}\leq 	\sum_{\substack{a,b>0\\a^2+b^2\geq 5}}\int_{Q_{a,b}}\frac{1}{\abs{x}^k} dx\leq \int_{\substack{\Re x\geq0\\\Im x\geq 0\\\abs{x}\geq 1}}\frac{1}{\abs{x}^k} dx=\frac{\pi}{2(k-2)}
	\end{equation}
	where $\Re x$ represents the real part of the complex number $x$ and $\Im x$ represents the imaginary part.
	Moreover, 
	\begin{equation}\label{ak}
	\sum_{a^2\geq 5}\frac{1}{\abs{a}^k}= 2\sum_{a=3}^\infty\frac{1}{a^k}\leq 2\int_{2}^\infty \frac{1}{a^k}da=\frac{2}{2^{k-1}(k-1)}.
	\end{equation}
	 Finally, using (\ref{intxk}) and (\ref{ak}),
	\begin{align*}
	\sum_{\omega\in \Lambda^*}\abs{\frac{1}{\omega^k}}&=\sum_{a^2+b^2\geq 5}\frac{1}{\abs{a+ib}^k}\\&=	\sum_{a^2\geq 5}\frac{1}{\abs{a}^k}+\sum_{b^2\geq 5}\frac{1}{\abs{b}^k}+	\sum_{\substack{a,b\neq 0\\a^2+b^2\geq 5}}\frac{1}{\abs{a+ib}^k}\\&=\sum_{a^2\geq 5}\frac{1}{\abs{a}^k}+\sum_{b^2\geq 5}\frac{1}{\abs{b}^k}+	4\sum_{\substack{a>0\\b> 0\\a^2+b^2\geq 5}}\frac{1}{\abs{a+ib}^k}\\&\leq  \frac{2}{2^{k-1}(k-1)}+\frac{2}{2^{k-1}(k-1)}+4\int_{\substack{\Re x\geq 0\\\Im x\geq 0\\\abs{x}\geq 1}}\frac{1}{\abs{x}^k}dx\\&\leq\frac{1}{k-1}+2\pi \frac 1{k-2}\\&\leq \frac{21}{k}
	\end{align*}
	where the last inequality follows from the fact that $k\geq 3$.
\end{proof}
 Given a complex number $x$ with $\abs{x}<1$, we define 
 \[
 \log(1-x)=-\sum_{i=1}^\infty\frac{x^i}{i}.
 \]
 \begin{lemma}\label{sigma}
 	Let $z\neq 0\in\C$ be such that the element of $\Lambda$ closest to $z$ is $0$. Suppose that $nz\in \Lambda$. Hence,
 	\[
 	\frac{0.14}{n}\leq \abs{\sigma(z)}\leq 4.04.
 	\]
 \end{lemma}
\begin{proof}
	Observe that $z=(a+ib)/n$ with $a$ and $b$ two integers such that $\abs{a}\leq n/2$ and $\abs{b}\leq n/2$.
	
  We start by finding a bound for the product of $\sigma$ considering only the terms in $\Lambda^*$. Put \[z_1:=\prod_{\omega\in \Lambda^*}\Big(1-\frac z\omega\Big)e^{\frac z\omega+\frac 12(\frac z\omega)^2}\] and observe that
  \begin{align*}
  \log\Big[\Big(1-\frac z\omega\Big)e^{\frac z\omega+\frac 12(\frac z\omega)^2}\Big]&=\log\Big(1-\frac zw\Big)+\log\Big(e^{\frac z\omega+\frac 12(\frac z\omega)^2}\Big)\\&=-\Big(\sum_{i=1}^\infty \frac{z^i}{i\omega^i}\Big)+\frac z\omega+\frac 12\Big(\frac z\omega\Big)^2\\&=-\sum_{i=3}^\infty \frac{z^i}{i\omega^i}.
  \end{align*} Taking the logarithm,
\begin{align*}
\abs{\log z_1}=\abs{\log \prod_{\omega\in\Lambda^*}(1-\frac{z}{\omega})
	e^{\frac{z}{\omega}+\frac{z^2}{2\omega^2}}}=
\abs{\sum_{i=3}^\infty\sum_{\omega\in \Lambda^*}\frac{z^i}{i\omega^i}}
\leq\sum_{i=3}^\infty\sum_{\omega\in \Lambda^*}\abs{\frac{z^i}{i\omega^i}}.
\end{align*}
Since the element of $\Lambda$ closest to $z$ is $0$, then $\abs{z}\leq (2)^{-1/2}$ and therefore, using Lemma \ref{sum},
\begin{align*}
\sum_{i=3}^\infty\sum_{\omega\in \Lambda^*}\abs{\frac{z^i}{i\omega^i}}&\leq\sum_{i=3}^\infty\sum_{\omega\in \Lambda^*}\frac{(\sqrt{2})^{-i}}{i\abs{\omega}^i}\\&=
\sum_{i=3}^\infty\sum_{a^2+b^2\geq 5}\frac{(\sqrt{2})^{-i}}{i\abs{a^2+b^2}^{i/2}}\\&=\sum_{i=3}^\infty\frac 1{(\sqrt{2})^ii}\sum_{a^2+b^2\geq 5}\frac{1}{\abs{a^2+b^2}^{i/2}}\\&\leq\sum_{i=3}^\infty\frac{21}{2^{i/2}i^2} \\&\leq \frac{21}{2^{3/2}3^2}+\frac{21}{2^{4/2}4^2}+\frac{21}{2^{5/2}5^2}+\frac{21}{36\sqrt{2}^6}\sum_{i=0}^\infty\Big(\frac{1}{\sqrt{2}}\Big)^i\\&= \frac{21}{2^{3/2}3^2}+\frac{21}{2^{4/2}4^2}+\frac{21}{2^{5/2}5^2}+\frac{21}{36\sqrt{2}^6}\frac{\sqrt{2}}{\sqrt{2}-1}
\\&\leq 1.56
\end{align*}
and then
\[
\abs{\Re(\log z_1)}\leq \abs{\log z_1}\leq 1.56.
\]
Therefore,
\[
e^{-1.56}\leq e^{-\abs{\Re(\log z_1)}}\leq\abs{e^{\log z_1}}=\abs{z_1}\leq e^{\abs{\Re(\log z_1)}}\leq e^{1.56}
\]
since, for $x\in \C$,
\[
e^{\Re(x)}=\abs{e^x}.
\]
Let $\Lambda_1=\Lambda \setminus \Lambda^*\setminus \{0\}$. This is a set of $12$ complex numbers. Then, by direct computation,
\[
\sum_{\omega\in \Lambda_1}\frac{1}{\omega}=0,
\] 
\[
\sum_{\omega\in \Lambda_1}\frac{1}{\omega^2}=0
\]
and
\[
\prod_{\omega\in \Lambda_1}\frac{1}{\omega}=\frac{1}{64}.
\] 
Now we need to deal with $\prod_{\omega\in\Lambda_1}(z-\omega).$ Recalling that $0$ is the element of $\Lambda$ closest to $z$, then $-0.5\leq \Re z\leq 0.5$ and $ -0.5\leq \Im z\leq 0.5$. We have
\[
\min_{\substack{-0.5\leq \Re z\leq 0.5 \\ -0.5\leq \Im z\leq 0.5}} \abs{\prod_{\omega\in\Lambda_1}\frac{(z-\omega)}{\omega}}\geq 0.94.
\]
This follows from the calculation of the minimum of the absolute value of the polynomial, that can be seen as a real polynomial in two variables, writing $z$ as $x+iy$.
In the same way
\[
\max_{\substack{-0.5\leq \Re z\leq 0.5 \\ -0.5\leq \Im z\leq 0.5}}\abs{\prod_{\omega\in\Lambda_1}\frac{(z-\omega)}{\omega}}\leq 1.2.
\]
Observe that 
\[
\frac 1n\leq \abs{z}\leq \frac{1}{\sqrt{2}}
\]
since every $n$-torsion point of $\C/\Lambda$ is in the form $(a+ib)/n$ and $z\neq 0$. Moreover,
\begin{align*}
\sigma(z)&=z\cdot z_1\cdot\prod_{\omega\in\Lambda_1}\Big(\frac{\omega-z}{\omega}\Big)\cdot e^{\sum_{\omega\in\Lambda_1}\frac z\omega}\cdot e^{\sum_{\omega\in\Lambda_1}\frac {z^2}{2\omega^2}}\\&
=z\cdot z_1\cdot \prod_{\omega\in\Lambda_1}\Big(\frac{\omega-z}{\omega}\Big).
\end{align*}
Using all the inequalities before, we conclude
\[
\frac{0.19}{n}\leq\frac 1n\cdot  e^{-1.56} \cdot 0.94 \leq \abs{z\cdot z_1\cdot \prod_{\omega\in\Lambda_1}\Big(\frac{\omega-z}{\omega}\Big)}=\abs{\sigma(z)}
\]
and 
\[
\abs{\sigma(z)}=\abs{z\cdot z_1\cdot \prod_{\omega\in\Lambda_1}\Big(\frac{\omega-z}{\omega}\Big)}\leq \frac{1}{\sqrt{2}}\cdot e^{1.56}\cdot1.2\leq 4.04.
\]
\end{proof}
Now, we need to bound $\abs{\sigma(z)}$ for $z$ so that $nz\in \Lambda$ but with $z$ in a larger region of the complex plane compared to the previous lemma.
\begin{lemma}\label{sig1}
	Let $z=(a+ib)/n\notin\Lambda$ with $a$ and $b$ two integers such that $\abs{a}\leq n$ and $\abs{b}\leq n$. So,
	\[
	\abs{\sigma(z)}\geq \frac{1}{2.1\cdot 10^{16}\cdot n}.
	\]
\end{lemma}
\begin{proof}
	Let $\omega\in \Lambda$ be such that the element of $\Lambda$ closest to $z+\omega$ is $0$. The element $\omega$ can be written in the form $x+iy$ for $x$ and $y$ two integers so that $-1\leq x,y \leq 1$. Thanks to \cite[Exercise 6.4.e]{arithmetic},
	\[
	\sigma(z)=\pm \sigma(z+\omega)e^{-\eta(\omega)(z+\omega/2)}=\pm \sigma(z+\omega)e^{-z\eta(\omega)}e^{-\frac{\omega\eta(\omega)}{2}},
	\] 
	where $\eta(\omega)$ defined in \cite[Exercise 6.4.b]{arithmetic}. Since we are interested in the absolute value, the sign is not important. The function $\eta$ is linear (see \cite[Exercise 6.4.c]{arithmetic}) and then
	\[
	\abs{\eta(\omega)}\leq \abs{\eta(1)}+\abs{\eta(I)}.
	\]
	Using the command "elleta" of PARI/GP, it is possible to compute the value of eta, and we have $\abs{\eta(1)}\leq 3.142$ and $\abs{\eta(I)}\leq 9.426$. Recalling that $\abs{z}\leq \sqrt{2}$, we have
	\[
	\abs{-z\eta(\omega)}\leq \sqrt{2}(3.142+9.426)\leq 17.8
	\]
	and, in the same way,
	\[
	\abs{-\omega\eta(\omega)/2}\leq 17.8.
	\]
	So, using Lemma \ref{sigma},
	\[
	\abs{\sigma(z)}=\abs{\sigma(z+\omega)}{\abs{e^{-z\eta(\omega)}e^{-\omega\eta(\omega)}}}\geq \frac{0.14}{n}e^{-\abs{\omega\eta(\omega)}-\abs{z\eta(\omega)}}\geq \frac{1}{2.1\cdot 10^{16}\cdot n}
	\]
	since
	\[
	0.14\cdot e^{-17.8}\cdot e^{-17.8}\geq \frac{1}{2.1\cdot 10^{16}}.
	\]
\end{proof}
\begin{lemma}\label{n2}
	Let $P$ be a non-trivial $n$-torsion point in $E(\overline{\Q})$, where $E$ is defined by $y^2=x^3-15G_4x$. Then,
	\[
	\abs{x(P)}\leq n^2+53.
	\]
\end{lemma}
\begin{proof}
	We want to bound 
	\[
	\wp(z)=\frac{1}{z^2}+\sum_{\omega\in \Lambda\setminus 0}\frac{1}{(z-\omega)^2}-\frac{1}{\omega^2}
	\]
	for $z=(a+ib)/n\neq 0$ with $a$ and $b$ two integers such that $\abs{a}\leq n/2$ and $\abs{b}\leq n/2$. 
	 Using the previous notation, we observe that $\abs{\omega}>2\abs{z}$ for $\omega\in\Lambda^*$. So,
	 \[
	 \abs{z-2\omega}\leq 2\abs{\omega}+\abs{z}\leq 2\abs{\omega}+\frac{\abs{\omega}}{2}=\frac 52\abs{\omega}
	 \] 
	 and
	 \[
	 \abs{z-\omega}\geq  \abs{\omega}-\abs{z}\geq \abs{\omega}-\frac{\abs{\omega}}{2}=\abs{\frac{\omega}{2}}.
	 \]
	 Observe that $\abs{z}\leq \sqrt{2}/2$ and using Lemma \ref{sum} we have 
	\begin{align*}
	\abs{\sum_{\omega\in \Lambda^*}\frac{1}{(z-\omega)^2}-\frac{1}{\omega^2}}&= \abs{\sum_{\omega\in \Lambda^*}\frac{z(z-2\omega)}{\omega^2(z-\omega)^2}} \\&\leq\abs{z}\sum_{\omega\in\Lambda^*}\frac{\frac 52\abs{\omega}}{\frac{\abs{\omega}^4}{4}}\\&\leq 5\sqrt{2}\sum_{\omega\in \Lambda^*}\frac{1}{\abs{\omega}^3}\\&\leq 5\sqrt{2}\frac{21}{3}\\&\leq 50.
	\end{align*}
	Furthermore,
	\[
	\sum_{\omega\in \Lambda_1}\frac{1}{\omega^2}=0
	\]
	and 
	\[
	\abs{\max_{\substack{-0.5\leq \Re z\leq 0.5 \\ -0.5\leq \Im z\leq 0.5}}\sum_{\omega\in \Lambda_1}\frac{1}{(z-\omega)^2}}\leq 3
	\]
	by direct computation.
	Then, using $\abs{z}\geq 1/n$ since $z\neq 0$, we obtain
	\begin{align*}
	\abs{\wp(z)}&\leq \abs{\frac 1{z^2}}+\abs{\sum_{\omega\in \Lambda^*}\frac{1}{(z-\omega)^2}-\frac{1}{\omega^2}}+\abs{\sum_{\omega\in \Lambda_1}\frac{1}{(z-\omega)^2}}+\abs{\sum_{\omega\in \Lambda_1}\frac{1}{\omega^2}}\\&\leq n^2+50+3.
	\end{align*}
	We conclude by observing that every $n$-torsion point in $\C/\Lambda$ can be written as $z=a+ib$ with $\abs{a}\leq n/2$ and $\abs{b}\leq n/2$. So, every non-trivial $n$-torsion point $P$ of $E(\C)$ is equal to $\varphi(z)$ for $z=(a+ib)/n\neq 0$ with $\abs{a}\leq n/2$ and $\abs{b}\leq n/2$. In conclusion,
	\[
	\abs{x(P)}=\abs{\wp(z)}\leq n^2+53.
	\]
\end{proof}
\begin{lemma}\label{G4}
	Let $G_4$ be defined as before. Hence, \[1\leq \abs{15G_4}\leq 128.\]
\end{lemma}
\begin{proof}
 By direct computation,
\begin{equation}\label{134}
\sum_{\omega\in \Lambda_1}\omega^{-4}=\frac{13}4.
\end{equation}Thus, using Lemma \ref{sum},
\begin{align*}
\abs{15G_4}&\leq 15\abs{\sum_{\omega\in \Lambda_1}\omega^{-4}}+ 15\sum_{a^2+b^2\geq 5}\frac{1}{(a^2+b^2)^2} \\&\leq 15\cdot\frac{13}{4}+ 15\cdot\frac{21}{4}\\&\leq 128.
\end{align*}
Now, we focus on the lower bound. Observe that, using (\ref{134}),
\[
\sum_{\omega\in \Lambda_1}\omega^{-4}+\sum_{b^2\geq 5}\frac{1}{(ib)^4}+\sum_{a^2\geq 5}\frac{1}{(a)^4}=\sum_{\omega\in \Lambda_1}\omega^{-4}+4\sum_{b\geq 3}\frac{1}{b^4}\geq\sum_{\omega\in \Lambda_1}\omega^{-4}= \frac{13}4
\]
and
\[
\sum_{\omega\in \Lambda}\omega^{-4}=\Big(\sum_{\omega\in \Lambda_1}\omega^{-4}+\sum_{b^2\geq 5}\frac{1}{(ib)^4}+\sum_{a^2\geq 5}\frac{1}{(a)^4}\Big)+4\Big(\sum_{\substack{a,b>0\\a^2+b^2\geq 5}}\frac{1}{(a^2+b^2)^2}\Big).
\]
As we showed in the proof of Lemma \ref{sum},
\[
4\abs{\sum_{\substack{a,b>0\\a^2+b^2\geq 5}}\frac{1}{(a^2+b^2)^2}}\leq \frac{2\pi}{4-2}=\pi.
\]
So,
\begin{align*}
\abs{\sum_{\omega\in \Lambda}\omega^{-4}}&\geq\abs{\Big(\sum_{\omega\in \Lambda_1}\omega^{-4}+\sum_{b^2\geq 5}\frac{1}{(ib)^4}+\sum_{a^2\geq 5}\frac{1}{(a)^4}\Big)}-4\abs{\sum_{\substack{a,b>0\\a^2+b^2\geq 5}}\frac{1}{(a^2+b^2)^2}}\\&\geq \frac{13}{4}-\pi\\&\geq 0.1.
\end{align*}
Finally,
\[
\abs{15G_4}=15\abs{\sum_{\omega\in \Lambda}\omega^{-4}}\geq 1.
\]
\end{proof}
\begin{lemma}\label{Kk}
	 Let $n\geq 2$ and $K_n$ be as in (\ref{eqKk}). So, we have
	\[
	K_n\geq \frac{1}{2.5\cdot 10^{36}\cdot n^6}.
	\]
\end{lemma}
\begin{remark}
	This bound is not sharp. Computational evidence shows that we could do much better. Indeed, using the command "polroots" of PARI/GP we can effectively compute $K_n$ for $n$ small. It seems that we could take $K_n\geq 6/n^2$. Anyway, for our goal, the bound is good enough.
\end{remark}
\begin{proof}
Using \cite[Exercise 6.3]{arithmetic}, given $z_1$ and $z_2$ two points in $\C/\Lambda$, we have
\[
\wp(z_1)-\wp(z_2)=\frac{\sigma(z_1+z_2)\sigma(z_2-z_1)}{\sigma(z_1)^2\sigma(z_2)^2}.
\]
Let now $z_1$ and $z_2$ be two non-zero complex numbers such that the element of $\Lambda$ closest to $z_1$ and $z_2$ is $0$ and such that $nz_1$ and $nz_2$ belong to $\Lambda$. Hence, $z_1+z_2$ and $z_1-z_2$ satisfy the hypothesis of Lemma \ref{sig1}, if $z_1\neq \pm z_2$. So, using Lemma \ref{sigma},
\[
\abs{\sigma(z_i)}\leq 4.04
\]
for $i=1,2$ and using Lemma \ref{sig1}
\[
\abs{\sigma(z_1\pm z_2)}\geq \frac{1}{2.1\cdot 10^{16}\cdot n}.
\] Therefore, 
\begin{align*}
\abs{\wp(z_1)-\wp(z_2)}&=\abs{\frac{\sigma(z_1+z_2)\sigma(z_2-z_1)}{\sigma(z_1)^2\sigma(z_2)^2}}\\&\geq\frac{1}{(2.1)^2\cdot(10)^{32}\cdot(4.04)^4n^2}\\&\geq\frac{1}{1.18\cdot10^{35}\cdot n^2}
\end{align*}
if $z_1\neq \pm z_2$.

Recall that $E$ is the elliptic curve defined by the equation $y^2=x^3-15G_4x$. Given $T_1$ and $T_2$ two non-trivial $n$-torsion points on the curve $E$, there are $z_1$ and $z_2$ as before such that $\varphi(z_1)=T_1$ and $\varphi(z_2)=T_2$. If $T_1\neq \pm T_2$, then we obtain
\[
\abs{x(T_1)-x(T_2)}=\abs{\wp(z_1)-\wp(z_2)}\geq\frac{1}{1.18\cdot10^{35}\cdot n^2}.
\]
Let $E_1$ be the elliptic curve defined by $y^2=x^3+x$ and so, if $T_1$ and $T_2$ are two non-trivial $n$-torsion points for $E_1$, then, thanks to the work in Lemma \ref{rootsep} and \ref{G4},
\begin{align*}
\abs{x(T_1)-x(T_2)}&\geq\frac{1}{\sqrt{15G_4}\cdot 1.18\cdot 10^{35}\cdot n^2}\\&\geq\frac{1}{\sqrt{128}\cdot 1.18\cdot 10^{35}\cdot n^2}\\&\geq  \frac{1}{1.4 \cdot10^{36}\cdot n^2}.
\end{align*}
Let $T$ be a non-trivial $n$-torsion point on $E_1$. Using again the work in the proof of Lemma \ref{rootsep}, we know that
\[
\abs{x(T)}\leq \frac{\max_{R\in E(\overline{\Q})[n]\setminus \{O\}}\abs{x(R)}}{\sqrt{15G_4}}.
\]
Thanks to Lemma \ref{n2} and \ref{G4} we have
\[
\abs{x(T)}\leq \frac{n^2+53}{1}.
\]
If $T_1$ and $T_2$ are two non-trivial $n$-torsion points in $E_1(\overline{\Q})$ with $T_1\neq \pm T_2$, then
\begin{align*}
\abs{x(T_1)^{-1}-x(T_2)^{-1}}&=\frac{\abs{x(T_1)-x(T_2)}}{\abs{x(T_1)x(T_2)}}\\&\geq \frac 1{1.4\cdot10^{36}\cdot n^2(n^2+53)^2}\\&\geq \frac{1}{2.5\cdot 10^{36}\cdot n^6}.
\end{align*}
The last inequality holds only if $n\geq 13$. Here we assumed $x(T_1)x(T_2)\neq 0$. For the cases $2\leq n\leq 13$ we prove the lemma computing effectively the constant $K_n$ using the command "polroots" of PARI/GP. For example, the roots of $\psi_3(x)$ when $a=1$ are $\pm 0.0.3933\dots$ and $\pm 1.46789\dots$. So, $K_3\sim 0.7866$ and then
\[
K_3\geq 0.75\geq \frac{1}{2.5\cdot 10^{36}\cdot 3^6}.
\]
\end{proof}
Thanks to Lemma \ref{2p-2} and \ref{split} we have that $\psi_p$ is reducible. Therefore, we would like to apply Proposition \ref{Bmk}. In order to do so, we need to verify the hypothesis of the proposition. For this reason, we prove the following two lemmas.
\begin{lemma}\label{mp+2}
	Fix $p\geq 13$ a prime congruent to $1$ modulo $4$ and let $P$ be a non-torsion point in $E_a(\Q)$.
	Suppose $m\geq p+2$. Then, 
	\[
	2m^2\geq \frac{2C+\log2\abs{a}^{1/2}-\log K_p}{\hat{h}(P)}.
	\]
\end{lemma}
\begin{proof}
	Using Lemma \ref{C} and the bound of $K_p$ in Lemma \ref{Kk},
	\[
	\frac{2C+\log2\abs{a}^{1/2}-\log K_p}{\hat{h}(P)}\leq \frac{6\log p}{\hat{h}(P)}+\frac{\log \abs{a}+85.1}{\hat{h}(P)}.
	\]
	Thanks to Lemma \ref{h}, \ref{100} and \ref{min} we have
	\[
	\frac{6\log p}{\hat{h}(P)}+\frac{\log \abs{a}+85.1}{\hat{h}(P)}\leq \frac{3\log p}{5}+1039.7.
	\]
	Observe that, for $p\geq 23$,
	\[
	2(p+2)^2\geq \frac{3\log p}{5}+1039.7.
	\]
	Therefore,
	\[
	2m^2\geq2(p+2)^2\geq \frac{3\log p}{5}+1039.7\geq
	 \frac{2C+\log2\abs{a}^{1/2}-\log K_p}{\hat{h}(P)}.
	\]
	If $p=13$, then computing $K_{13}$ with PARI/GP we have $K_p\geq 0.04$. So,
	\[
	\frac{2C+\log2\abs{a}^{1/2}-\log K_p}{\hat{h}(P)}\leq \frac{\log \abs{a}+4.5}{\hat{h}(P)}
	\]
	and using the usual inequalities
	\[
	\frac{\log \abs{a}+4.5}{\hat{h}(P)}\leq 63.
	\]
	Therefore,
	\[
	2m^2\geq 2(p+2)^2=450\geq 63\geq \frac{2C+\log2\abs{a}^{1/2}-\log K_p}{\hat{h}(P)}.
	\]
	The case $p=17$ is analogous.
\end{proof}
\begin{lemma}\label{0.8}
	Let $n$ be an integer such that the smallest divisor of $n$ is $p\geq 13$. Hence,
\[
\frac{0.8}{p} \leq \frac{2p-2}{p^{2}}-\rho(n).
\]
\end{lemma}
\begin{proof}
Since the smallest prime divisor of $n$ is $p$,
\begin{align*}
\frac{2p-2}{p^{2}}-\rho(n)
& \geq \frac{2p-2}{p^{2}}-\frac{1}{p^{2}}-\sum_{n \geq p+2} \frac{1}{n^{2}} \\
& \geq \frac{2p-3}{p^{2}}-\int_{p+1} \frac{dx}{x^{2}} \\
&= \frac{2p-3}{p^{2}}-\frac{1}{p+1} \\
&\geq \frac{0.8}{p},
\end{align*}
where the last inequality holds since $p \geq 13$.
\end{proof}
Now, we are ready to prove Theorem \ref{thm2}. Fix a prime $p\equiv 1 \mod{4}$. We want to use Proposition \ref{Bmk}. We know that $\psi_p$ is reducible and the hypothesis of the proposition is satisfied thanks to Lemmas \ref{mp+2} and \ref{0.8}. So, if $B_{mp}$ does not have a primitive divisor, then inequality (\ref{2mk}) must hold. Hence, in order to prove the theorem, we will show that this inequality does not hold, using the bound on $K_n$.
\begin{proof}[Proof of Theorem \ref{thm2}]
	Suppose now that $n=mp$ is a square-free positive integer with $p \equiv 1 \bmod 4$
	and $m$ such that the smallest divisor of $n$ is larger than $p$. We want to
	show that $B_{n}$ always has a primitive divisor. So we may assume that $n$ is odd, and since $n$ is square-free, it follows that $m \geq p+2$. Furthermore, the case of $p=5$ is handled by Theorem \ref{Teo} here, so we can assume that $p \geq 13$. We will assume that $B_{n}$ does not have a primitive divisor and obtain a contradiction.
	
	Since $p$ is odd, from Lemmas \ref{res} and \ref{gn}, we have
	\[
	g_{p}\mid \left( 2\Delta^{2} \right)^{\frac{p^{2}-1}{4}}
	=\left( 2^{13}a^{6} \right)^{\frac{p^{2}-1}{4}}.
	\]
	Thanks to Lemma \ref{2p-2} and \ref{split}, we know that $\psi_p$ is reducible. 
	Under our conditions here, Lemma \ref{mp+2} holds and hence we can apply Proposition \ref{Bmk},
	along with the above upper bound for $g_{p}$, to obtain
	\begin{eqnarray}
	\label{eq:10}
	2(mp)^{2} \left( \frac{d}{p^{2}} - \rho(mp) \right) \hat{h}(P)
	& \leq & \frac{p^{2}-1}{4} \log \left( 2^{13} a^{6} \right) - d \log \left( \frac{K_{p}|a|^{-1/2}}{2} \right) \nonumber \\
	&      & +2dC+2\log(mp)+2C\omega(mp).
	\end{eqnarray}
	Here we can take $d=2p-2$, thanks to Lemma \ref{2p-2}.
	Also, from Lemma \ref{Kk}, we find that the left-hand side of \eqref{eq:10} is at most
	\begin{align*}
	&\frac{p^{2}-1}{4}\log \left( (2^{13}|a|^{6})  \right) + 4(p-1)C+2(\log(n)
	\\+& C\omega(n))+2(p-1) \log \Big( 5\cdot 10^{36}p^{6} \sqrt{a} \Big).
	\end{align*}
	It will help us in what follows to express the right-hand side of this inequality
	solely in terms of $a$, $m$ and $p$. To do so, we will collect together like
	terms. Since $\omega(n) \leq \log(mp)/\log(p)=\log(m)/\log(p)+1 \leq \log(m)/\log(13)+1$
	and by the definition of $C$ in Lemma \ref{C}, we obtain
	\begin{align*}
	2C\omega(n)+2\log(n)
	 \leq& \Big(0.52+\frac{\log |a|}{2}\Big) \left( \frac{\log(m)}{\log(13)}+1 \right)+2\log(m)+2\log(p) \\
	=& \log(m) \left( \frac{0.52}{\log(13)}+2 +\frac{\log|a|}{2\log(13)} \right)\\&+0.52+\frac{\log|a|}{2}+2\log(p),
	\end{align*}
	and
	\begin{eqnarray*}
		&      & 4(p-1)C+\frac{p^{2}-1}{4} \log \left( 2^{13}|a|^{6} \right)+2(p-1) \log \left( 5\cdot 10^{36}p^{6} \sqrt{a} \right) \\
		& \leq & 4(p-1)(0.26+\log |a|/4) + \frac{p^{2}-1}{4} \log \left( 2^{13} \right) + (3/2)\left( p^{2}-1 \right) \log |a| \\
		&      & +12(p-1)\log(p) + (p-1)\log|a| + 2(p-1) \log \left( 5\cdot 10^{36} \right) \\
		&   =  & \log|a| \left( (3/2)\left( p^{2}-1 \right) + 2(p-1) \right) + \frac{p^{2}-1}{4} \log \left( 2^{13} \right)\\
		&      &  + 12(p-1)\log(p) + 2(p-1) \log \left( 5\cdot 10^{36} \right)
		+ 1.04(p-1).
	\end{eqnarray*}
	
	Combining these two inequalities, we find that
	\begin{eqnarray*}
		&      & 2m^{2}p^{2} \hat{h}(P) \left( \frac{2p-2}{p^{2}} - \rho(n) \right)\\
		& \leq & \log(m) \left( 2.21+0.2\log|a| \right)
		+\log|a| \left( (3/2)\left( p^{2}-1 \right) + 2p-3/2 \right) \\
		&      & +2.26\left( p^{2}-1 \right)
		+(12p-10)\log(p) + 170.1(p-1)+0.52.
	\end{eqnarray*}
	
	Since $p\geq 13$, with some basic analysis we obtain
	\begin{align*}
	(3/2)\left( p^{2}-1 \right) + 2p-3/2 &< 1.64p^{2} \text{ and} \\
	2.26\left( p^{2}-1 \right)+(12p-10)\log(p) + 170.1(p-1)+0.52 &< 16.55p^{2}.
	\end{align*}
	Therefore,
	\[
	2m^{2}\Big( \frac{2p-2}{p^{2}} - \rho(n) \Big)\\
	 \leq \frac{\log(m) \left( 2.21+0.2\log|a| \right)
	+p^2(1.64\log|a|+16.55)}{\hat{h}(P)p^2}.
	\]
	Using the inequalities for $\hat{h}(P)$ in Lemmas \ref{h} and \ref{100},
	\[
	\frac{0.2\log|a|+2.21}{\hat{h}(P)} \leq 28.7
	\]
	and
	\[
	\frac{1.64\log|a|+16.55}{\hat{h}(P)} \leq 214.4.
	\]
	
	Therefore, in order for the inequality in (\ref{eq:10}) to hold, we must also have
	\begin{equation}
	\label{eq:11}
	2m^{2}\left( \frac{2p-2}{p^{2}} - \rho(n) \right)
	\leq 28.7\frac{\log(m)}{p^{2}}+214.4.
	\end{equation}
	
	 Using Lemma \ref{0.8}, we obtain the inequality
	\begin{equation}
	\label{m}
	\frac{1.6m^{2}}{p}
	\leq 2m^{2}\left( \frac{2p-2}{p^{2}} - \rho(n) \right)
	\leq 28.7\frac{\log(m)}{p^{2}}+214.4.
	\end{equation}
	
	If we show that this inequality does not hold for some $m$ and $p$, then $B_{mp}$ always has a primitive divisor.
	If we fix $p\geq 13$, then the function
	\[
	\frac{1.6m^2}p-28.7\frac{\log m}{p^2}-214.4
	\]
	is increasing for $m\geq p+2$. Take $p> 129$, then
	\[
	\frac{1.6m^2}p-28.7\frac{\log m}{p^2}-214.4\geq \frac{1.6(p+2)^2}p-28.7\frac{\log(p+2)}{p^2}-214.4\geq 0
	\]
	where the last inequality follows from the assumption $p> 129$. So, there is always a primitive divisor for $p> 129$ and $m\geq p+2$ since (\ref{m}) does not hold. If $m\geq p+34$ and $p\geq 13$, then we have
	\[
		\frac{1.6m^2}p-28.7\frac{\log m}{p^2}-214.4\geq\frac{1.6(p+34)^2}p- 28.7\frac{\log(p+34)}{p^2}-214.4\geq 0.
	\]
	Therefore, it remains to check only finitely many cases in the form $p\cdot m$ with $p\leq  129$ and $p<m<p+34$. Observe that, since $p$ is the smallest divisor of $n$, then in these cases $m$ is prime. Hence, it remains to check the 85 cases where $n=pq$ with $p$ and $q$
	primes satisfying $13\leq p\leq 129$ with $p\equiv 1 \mod{4}$ and $p<q< p+34$. For these cases we substitute the values in (\ref{eq:11}), taking $m=q$ and we obtain 
	\[
	2q^2\Big(\frac{2p-3}{p^2}-\frac{1}{q^2}\Big)\leq 28.7\frac{\log q}{p^2}+214.4.
	\]
	If $n\neq 13\cdot 17$, $13\cdot19$, $13\cdot 23$, $17\cdot 19$, $ 17\cdot 23$, $17\cdot 29$, $17\cdot 31$, $29\cdot 31$, $29\cdot 37$ $37\cdot 41$, $37\cdot 43$, $41\cdot 43$ or $41\cdot 47$ then this inequality does not hold and so $B_n$ has always a primitive divisor. 
	
	Now, we show how to deal with the case $n=13\cdot 17$. Using the command "polroots" of PARI/GP, we can compute the constant $K_{13}$. We have $K_{13}\geq 0.04$. We substitute the value $n=13\cdot 17$ in (\ref{eq:10}) using this bound and we obtain
	\[
	\hat{h}(P)\leq \frac{497+277\ln \abs{a}}{12956}.
	\]
	Using again the usual inequalities on $\hat{h}$, we have that this inequality does not hold and then $B_{13\cdot 17}$ has always a primitive divisor. The other cases are analogous.
\end{proof}
\begin{remark}
	An analogue of Theorem \ref{thm2} can also be obtained for any elliptic curve, $E$, with complex multiplication and for all the primes that split in $\End(E)$. In order to find uniform bounds, it is necessary to have inequalities similar to the inequalities in Lemma \ref{h} and \ref{100}.
\end{remark}

\begin{remark}
	We cannot apply the technique used in the proof of the last theorem in order to prove an analogue for the term in the form $B_p$ with $p\equiv 1\mod 4$. Indeed, substituting $m=1$ in the inequality (\ref{eq:10}) we would like to show that the inequality does not hold. But the LHS grows as $p^2$ and the RHS grows as $p$, so the inequality holds for $p$ large enough. Hence, for $p$ large this method does not work.
\end{remark}

\section*{Acknowledgements}
The author would like to thank the anonymous referee for
their careful reading and numerous suggestions to improve this paper.

\normalsize
\baselineskip=17pt
\bibliographystyle{siam}
\bibliography{biblio}
	MATTEO VERZOBIO, UNIVERSIT\'A DI PISA, DIPARTIMENTO DI\\ MATEMATICA, LARGO BRUNO PONTECORVO 5, PISA, ITALY\\
\textit{E-mail address}: matteo.verzobio@gmail.com

\end{document}